\documentclass[10pt]{article}%
\usepackage{amsmath}
\usepackage{amssymb}
\usepackage{amsbsy}
\usepackage{amsthm}
\usepackage{epsfig}
\usepackage{wrapfig}
\usepackage{xcolor}
\usepackage{graphicx}
\usepackage{colordvi}
\usepackage{color}
\usepackage{graphics}
\usepackage[left=1in, top=1in, right=1in, bottom=1in]{geometry}
\usepackage{amsfonts}
\usepackage{verbatim}
\usepackage{tabularx}
\usepackage{subfigure}
\usepackage{wrapfig}
\usepackage{floatrow}
\usepackage{esvect}  %arrow of vector
\usepackage[utf8]{inputenc}  % for gather

\numberwithin{equation}{section}

\newtheorem{theorem}{Theorem}[section]

\newtheorem{lemma}{Lemma}[section]

\newtheorem{cor}{Corollary}[section]
\newtheorem{definition}{Definition}[section]

\newtheorem{remark}{Remark}[section]

\newcommand{\abs}[1]{\left\vert#1\right\vert}

\newcommand{\bff}{{  f}}

\newcommand{\bu}{{\bf u}}

\newcommand{\bX}{{X }}

\usepackage{accents}

\begin{document}

\date{}

\author{Mustafa Aggul \footnotemark[1]\, \hspace{.1in}Fatma G. Eroglu \footnotemark[2] \hspace{.1in} Song\"{u}l Kaya\footnotemark[3]\hspace{.1in} Alexander E. Labovsky\footnotemark[4] }

\title{A projection based Variational Multiscale Method for Atmosphere-Ocean Interaction}

\maketitle

\renewcommand{\thefootnote}{\fnsymbol{footnote}}

\footnotetext[1]{Department of Mathematics, Hacettepe University, Ankara 06800, Turkey, email: mustafaaggul@hacettepe.edu.tr}
\footnotetext[2]{Department of Mathematics, Faculty of Science, Bart{\i}n University, Bart{\i}n 74110, Turkey, email: fguler@bartin.edu.tr}

\footnotetext[3]{Department of Mathematics, Middle East Technical University, Ankara 06800, Turkey, email: smerdan@metu.edu.tr}

\footnotetext[4]{Department of Mathematical Sciences, Michigan Technological University, Houghton, MI 49931, USA, email: aelabovs@mtu.edu
}

%\footnotetext[4]{Department of Mathematical Sciences, Clemson
%University, Clemson SC 29634 USA and Department of Mathematical Sciences, Michigan Tech University, Houghton MI 49931; email: newilson@mtu.edu; partially supported by NSF grants DMS 0914478 and DMS 1112593.}

\begin{abstract}
The proposed method aims to approximate a solution of a fluid-fluid interaction problem in case of low viscosities. The nonlinear interface condition on the joint boundary allows for this problem to be viewed as a simplified version of the atmosphere-ocean coupling. Thus, the proposed method should be viewed as potentially applicable to air-sea coupled flows in turbulent regime. The method consists of two key ingredients. The  geometric averaging approach is used for efficient and stable decoupling of the problem, which would allow for the usage of preexisting codes for the air and sea domain separately, as ``black boxes''. This is combined with the variational multiscale stabilization technique for treating flows at high Reynolds numbers. We prove the stability and accuracy of the method, and provide several numerical tests to assess both the quantitative and qualitative features of the computed solution.

\end{abstract}

\section{Introduction}

The study of solving coupled Navier-Stokes equations with special interface conditions
is of considerable interest, for instance in the simulation of atmosphere-ocean (AO)
interaction or two layers of a stratified fluid. In this paper, we investigate a low-viscosity fluid-fluid interaction problem, aiming at modeling AO flow in a turbulent regime.

Consider the $d$-dimensional ($d=2,3$) polygonal or polyhedral domain $\Omega$ in space that consists of two subdomains $\Omega_1$ and $\Omega_2$, coupled across an interface $I$, for times
$t\in [0,T]$. Coupling problem is: given  $ \nu_{i}>0 $, $f_{i}:[0,T]\rightarrow H^{1}(\Omega_{i})^d,u_{i}(0)\in
H^{1}(\Omega_{i})^d $ and $\kappa \in \mathbb{R} $, find (for $i=1,2$) $u_{i}:\Omega_{i}\times [0,T]\rightarrow \mathbb{R}^d$ and $p_{i}: \Omega_{i}\times [0,T]\rightarrow \mathbb{R}$ satisfying (for $0<t\leq T$)
%\begin{eqnarray}
 %\begin {array}{rcll}
%\partial_{t}u_i +u_i \cdot \nabla u_i - \nu_i\Delta u_i + \nabla p_i &=&  f_i&\mathrm{in}\ \Omega_i, \label{eq:atmo} \\
%-\nu_i \hat{n}_i  \cdot \nabla u_i \cdot \tau &=& \kappa |u_i - u_j|( u_i - u_j )\cdot \tau &
%\mathrm{ on }\ I \ \mathrm{for} \ i=1,2,\ i \neq j \,  , \label{eq:atmoI}\\
%u_i \cdot \hat{n}_i  &=&0& \mathrm{ on }\ I \ \mathrm{for} \ i=1,2,\quad\label{eq:atmoNONLINEAR}\\
%\nabla \cdot u_i &=&0 &\mathrm{ in }\ \Omega_i,  \label{eq:atmoNONLINEAR1}\\
%u_i(x,0)&=&u_i^0(x)& \mathrm{ in }\ \Omega_i , \label{eq:atmoIC}\\
%u_i&=&0 &\mathrm{ on }\ \Gamma_i = \partial\Omega_i\setminus I , \label{eq:atmoBC}
%\end{array}
%\end{eqnarray}

\begin{eqnarray}
\partial_{t}u_i -\nu_i\Delta u_i+u_i \cdot \nabla u_i + \nabla p_i &=&  f_i \qquad \mathrm{in}\ \Omega_i, \label{eq:atmo} \\
-\nu_i \hat{n}_i  \cdot \nabla u_i \cdot \tau &=& \kappa |u_i - u_j|( u_i - u_j )\cdot \tau \quad \mathrm{on}  \ I \ \mathrm{for} \ i,j = 1,2, \ i \neq j \,  , \label{eq:atmoI}\\
u_i \cdot \hat{n}_i  &=&0 \qquad \mathrm{on} \ I \ \mathrm{for} \ i,j = 1,2, \label{eq:atmoNONLINEAR}\\
\nabla \cdot u_i &=&0 \qquad  \mathrm{ in } \ \Omega_i,  \label{eq:atmoNONLINEAR1}\\
u_i(x,0)&=&u_i^0(x) \qquad\mathrm{ in }\ \Omega_i , \label{eq:atmoIC}\\
u_i&=&0 \qquad\mathrm{ on } \ \Gamma_i = \partial\Omega_i\setminus I , \label{eq:atmoBC}
\end{eqnarray}
where $|\cdot|$ represents the Euclidean norm and the vectors $\hat{n}_i$ are the unit normals on $\partial \Omega_i$, and $\tau$ is any vector such that $\tau\cdot\hat{n}_i=0$. Here $u_i$, and $p_i$ denote the unknown velocity fields and pressure. The parameters are $\nu_i$ kinematic viscosities, $f_i$ the body forcing on the velocity, $\kappa$ the friction parameter (frictional drag force is assumed to be proportional to the square of the jump of the velocities across the interface).

Numerical methods for solving this type of coupled problems in laminar flow regime have been investigated \cite{BK06,CHL12,ZHS16, ACEL18}. In \cite{CHL12}, IMEX and geometric averaging (GA) time stepping methods have been proposed (and further developed in \cite{ACEL18}) for the Navier-Stokes equations with nonlinear interface condition.

The study of AO interaction has received considerable interest in the last thirty years, starting with the seminal paper of Lions, Temam and Wang, \cite{LTW1,LTW2}, on the analysis of full equations for AO flow. Today, many models exist and an abundance of software code is available for climate models (both global and regional), hurricane propagation, coastal weather prediction, etc. see, e.g., \cite{BWCK00, BKLG96, PSSB07} and references therein. The reasoning behind most of these models is as follows: the boundary condition on the joint AO interface must be chosen in such a way, that fluxes of conserved quantities are allowed to pass from one domain to the other. In particular, the nonlinear interface condition (\ref{eq:atmoI}), together with
(\ref{eq:atmoNONLINEAR}) ensures that the energy is being passed between the two domains in the model above, with the global energy still being conserved.

The AO coupling problem (as well as its modest version, the fluid-fluid interaction with nonlinear coupling, considered in this report) provides many challenges. In addition to the usual issues one has to overcome when solving the Navier-Stokes equations, the AO models should allow to use different spacial and temporal scales for the atmosphere and ocean domains, as the energy in the atmosphere remains significant at smaller time scales and larger spatial scales, than the energy of the ocean. In order to do so, as well as make use of the existing codes written separately for the fluid flows in the air or the ocean domains, one needs to create partitioned methods, that allow for a stable and accurate decoupling of the AO system.

The literature on numerical analysis of time-dependent coupling problem \eqref{eq:atmo}-\eqref{eq:atmoBC} is somewhat scarse; some approaches to creating a stable, accurate, computationally attractive decoupling method can be found in \cite{BK06, CHL12, CG11, LBD15, ZHS16, ACEL18}. The methods in \cite{LBD15, ACEL18} provide second order accuracy in temporal discretization. However, the authors could not find any reports on methods for approximating the solution of \eqref{eq:atmo}-\eqref{eq:atmoBC} in a turbulent regime. This problem is magnified over the usual issues in turbulence modeling, because of several extra obstacles: the size of the problem, the necessity to treat the atmosphere and ocean codes as ``black boxes'' - therefore utilizing one of only a few existing decoupling methods; and, finally, the lack of benchmark problems for turbulent AO coupling. We propose to start working in this direction by using a stabilization technique for low-viscosity problem \eqref{eq:atmo}-\eqref{eq:atmoBC}.

Various stabilizations have proven to be essential computational tools for the numerical simulations.The general idea of two level stabilization is pioneered by Marion and Xu in \cite{MJ94} and  the analysis for Navier-Stokes is presented in seminal papers \cite{G01,LG01}. This idea has been strongly connected with variational multiscale (VMS) methods were introduced in \cite{Hu95,Gue99}. VMS methods have proven to be an accurate and systematic approach to the numerical simulation of multiphysics flows and different realizations of VMS in the literature exist, e.g., see\cite{G04,RC04,jokaya1,jokaya3}.  In particular, we consider a projection-based VMS in this paper which has been proposed in \cite{L02}. According to VMS concept, global stabilization is introduced in all scales, then removes the effective stabilization on the large scales of the solution. In this way, stabilization is effective only on the smallest scales, where the non-physical oscillations occur. For more details, we refer the reader to \cite{jokaya1,jokaya3}. We also refer to \cite{V17} for the derivation of the different VMS methods for turbulent flow simulations.

Due to the success of VMS method, there is a natural desire to introduce this accurate and systematic approach to the simulation of atmosphere-ocean interaction. We consider an extension of VMS method with GA of the nonlinear interface condition. As first contribution of this paper, we first show the conservation of GA-VMS method's discrete kinetic energy, frequently evaluated quantity of interest in AO flow simulations along with stability and long-time stability properties of GA-VMS method. We show both stability bounds are unconditional, i.e., without any restriction on time step size. Secondly, we provide a precise analysis of the stability, convergence and accuracy of the GA-VMS method. Lastly, we present numerical studies in case of different viscosities compared with monolithically coupled algorithms.

The paper is organized as follows. The GA-VMS method for solving \eqref{eq:atmo}-\eqref{eq:atmoBC} in the case of high Reynolds number(s) is presented in Section 2, along with a short discussion on an alternative formulation of the method. After mathematical preliminaries are introduced in Section 3, a complete  numerical analysis is then done on the proposed method in Section 4. Finally, Section 5 provides the numerical tests that validate the theoretical findings, and conclusions are given in Section 6.

\section{GA-VMS method for atmosphere-ocean interaction problem }

In this paper, standard notations of Lebesgue and Sobolev spaces are used. The space $(L^2(\Omega))^d$ is equipped with the inner product, $(\cdot,\cdot)$ and the norm $\|\cdot\|$. In particular, the norm $L^3(I)$ at the interface will be denoted by $\|\cdot\|_I$. The Hilbert space $(H^k(\Omega))^d$ is equipped with the norm $\|\cdot\|_k$. The norm of the dual space of $(H^{-1}(\Omega))$ of $(H_0^1(\Omega))$ is denoted by $\|\cdot\|_{-1,\Omega}$. The other norms are labeled with subscripts.

For the weak formulation of problem \eqref{eq:atmo}-\eqref{eq:atmoBC}, we use the function spaces for $i=1,2$
\begin{eqnarray}
&X_{i}&  :=\{v\in (L^2(\Omega_{i}))^d:\nabla v \in L^2(\Omega_{i})^{d\times d}, \, v= 0\,\mbox{  on  }\,\partial \Omega{_{i} \backslash I, \hspace{0.1in} v\cdot \hat{n}_i = 0 \mbox{  on  } I}\},
%X &=& (H_0^1(\Omega))^d :=\{v\in (L^2(\Omega))^d:\nabla v \in L^2(\Omega)^{d\times d}, \, v= 0\,\mbox{on}\,\partial \Omega\},
\nonumber
\\
&Q_{i}& = L_0^2(\Omega_{i}) := \{q \in L^2(\Omega_{i}): \int_{\Omega_{i}} q \ dx = 0\}.  \nonumber
\end{eqnarray}
Herein, define $X=X_1\times X_2$  and $Q=Q_1\times Q_2$. For $u_i\in X_i$ and $q_i\in Q_i$, we denote ${\bf u}=(u_1,u_2)$ and ${\bf q}=(q_1,q_2)$, respectively.

Using these function spaces, the weak formulation of \eqref{eq:atmo}-\eqref{eq:atmoBC} is as follows: Find $(u_i,p_i)\in (X_i,Q_i)$ for $ i,j = 1,2, \ i \neq j$ such that for all $(v_{i},q_i) \in (X_i,Q_i)$
%The weak formulation of \eqref{eq:atmo}-\eqref{eq:atmoBC} is obtained by multiplying with test functions $(v_{h,i},q_i) \in ({X^h_i, Q^h_i})$ at time $t^{n+1}$ and integrating over the subdomains $\Omega_i$. Then the problem is as follows: Find $(u_i,p_i)\in (X_i,Q_i)$ for $ i,j = 1,2, \ i \neq j$ satisfying
\begin{eqnarray}
(\partial_{t}u_i ,v_{i})_{\Omega_i} +\nu_i(\nabla  {u}_{i},\nabla v_{i})_{\Omega_i}
+ { c_i({u}_{i};{u}_{i},v_{i})}
-(p_i,\nabla\cdot v_i)_{\Omega_i}+(\nabla \cdot u_i,q_i)_{\Omega_i} \nonumber
\\
+\kappa \int_{I}|u_i-u_j |(u_i-u_j)v_{i}ds=(\bff_i,v_{i})_{\Omega_i}. \label{weak1}
\end{eqnarray}
Here and in the rest of the paper, $c_i(\cdot; \cdot,\cdot)$ denotes the usual, explicitly skew symmetrized trilinear form
\begin{equation}
c_i(u;v,w)=\frac{1}{2}(u \cdot \nabla v,w)_{\Omega_i}-\frac{1}{2}(u \cdot \nabla w,v)_{\Omega_i}
\end{equation}
for functions $u,v,w \in X_i$, $i=1,2 $ on $\Omega_i$.  Notice the well known property
$$
c_i(u;v,w)_{\Omega_i}=-c_i(u;w,v)_{\Omega_i}
$$
for all $u,v,w \in X_i$ such that in particular $c_i(u;v,v)=0$ for all $u,v\in X_i$.

The standard monolithic weak formulation of \eqref{eq:atmo}-\eqref{eq:atmoBC} is obtained by summing
\eqref{weak1} over for $ i,j = 1,2, \ i \neq j$ and is to find
$({\bf u},p)\in (X,Q)$  such that for all $({\bf v},{\bf q}) \in (X,Q)$
\begin{eqnarray}
(\partial_{t}{\bf u} ,{\bf v}) +\nu(\nabla  {\bf u},\nabla {\bf v})
+ { c({\bf u};{\bf u},{\bf v})}
-({\bf p},\nabla\cdot {\bf v})+(\nabla \cdot {\bf u},{\bf q})+\kappa \int_{I}[{\bf u}][{\bf u}]{\bf v}ds =({\bf f},{\bf v})_{\Omega_i}, \label{weak2}
\end{eqnarray}
where $[\cdot]$ denotes the jump across the interface $I$ and ${\bf f}=f_i$, $\nu=\nu_i$ on $\Omega_i$.

For finite element discretization, let $T_{i}^h$ and $T_i^H$ be  admissible triangulations of $\Omega_i$, where  $T_{i}^h$ refers to fine mesh and $T_i^H$ denotes the coarse mesh. Let $(X_{i}^h, { Q^h_{i}}) \subset (X_i,Q_i)$ be  conforming finite element spaces satisfying the so-called discrete inf-sup condition \cite{GR79,G89}. In our tests, we have used the velocity-pressure pairs of spaces $(P_k,P_{k-1}),\,k\geq 2$.
Let $V_{i}^h$ be the space of the discretely divergence-free functions
\begin{eqnarray}
V_{i}^h=\{v_{h,i}\in X_i^h: (q_{h,i},\nabla \cdot v_{h,i})=0, \, for \,\, all \,\, q_{h,i} \in Q_{i}^h\},
\end{eqnarray}
which is a closed subspace of $X_i^h$. The dual space of $V_i^h$ is given by $V_i^{h*}$ with norm $\|\cdot\|_{V_i^{h*}}$.We also need to introduce the space
\begin{eqnarray}
L^{\infty}(\mathbb{R}^+,V_i^{h*})&=&\{f_i:\Omega_{i}^d\times\mathbb{R}^+\to \mathbb{R}, \exists M<\infty \, with \,\, \|f_i(t)\|_{V_i^{h*}}<M \, a.e.\, t>0\}
\end{eqnarray}

To solve two decoupled systems (atmosphere and ocean separately) through GA on the interface with the projection-based VMS formulation, let $L_i^H \subset (L^2(\Omega))^{d\times d} $ be a finite dimensional space of functions defined on $\Omega_i$ representing a coarse or large scale space and let  $\nu_T$  be eddy viscosity term assumed herein a non-negative function depending on the mesh size $h$.

We now present the projection-based VMS discretization of \eqref{weak1} by using the Euler method in time. For this purpose, consider a partition $0=t_0<t_1<\dots<t_M=T$ of the time interval $[0,T]$ and define $\Delta t=T/M$, $t_n=n\Delta t$. GA-VMS formulation applied to the problem \eqref{weak1} reads as follows: Find $(u_{h,i}^{n+1},p_{h,i}^{n+1},\mathbb{G}_i^{\mathbb{H},{n+1}})\in (X_{i}^h,Q_{i}^h,L_i^H)$ satisfying
\begin{eqnarray}
(\frac{{u^{n+1}_{h,i}}-u_{h,i}^{n}}{\Delta t} ,v_{h,i})_{\Omega_i} +(\nu_i+\nu_T)(\nabla  u_{h,i}^{n+1},\nabla v_{h,i})_{\Omega_i} +c_i(u_{h,i}^{n+1}; u_{h,i}^{n+1},v_{h,i})
-(p_{h,i}^{n+1},\nabla \cdot v_{h,i}) \nonumber
\\
 + (\nabla\cdot u_{h,i}^{n+1},q_{h,i})_{\Omega_i}+\kappa \int_{I}|[{\bf u}_h^n]|u_{h,i}^{n+1}v_{h,i}ds
-\kappa \int_{I}u_{h,j}^n|[{\bf u}_h^n]|^{1/2}|[{\bf u}_h^{n-1}]|^{1/2}v_{h,i}ds \nonumber
\\
=(\bff_i^{n+1},v_{h,i})_{\Omega_i}+\nu_T(\mathbb{G}_i^{\mathbb{H},n},\nabla v_{h,i}) \label{BE7}\\
(\mathbb{G}_i^{\mathbb{H},n}-\nabla u_{h,i}^n, \mathbb{L}_i^H)_{\Omega_i}=0,\label{alg4}
\end{eqnarray}
for all $(v_{h,i},q_{h,i},\mathbb{L}_i^H) \in (X_{i}^h,Q_{i}^h,L_i^H)$.

\begin{remark}
In \eqref{alg4}, the tensor $\mathbb{G}_i^{\mathbb{H},n}$ represents the large scales
of $\nabla u_{h,i}$, defined by $L^2$-projection of $\nabla u_{h,i}^n$ on $\Omega_i$ into the large scale space $L_i^H$ (see Definition \ref {defpro}). Hence, the difference $\mathbb{G}_i^{\mathbb{H},n}-\nabla u_{h,i}^n$  represents the resolved small scales. This way, the GA-VMS method \eqref{BE7}-\eqref{alg4} introduces the additional viscous term into the momentum equation acting only on the resolved small scales.
We note that the $L^2$- projection terms for $\mathbb{G}_i^{\mathbb{H},n}$ can be discretized implicitly or explicitly in time. We will consider here the computationally attractive explicit discretization, and refer the reader to \cite{jokaya3,JK08} for further discussions on explicit vs. implicit discretizations of $\mathbb{G}_i^{\mathbb{H},n}$.
\end{remark}
\begin{remark}
In GA-VMS formulation of \eqref{BE7}-\eqref{alg4}, the large scale spaces $ L_i^H$ and $\nu_T$ parameters must be chosen. The first approach is to define $L_i^H$  using in lower order finite element spaces on the same mesh, provided that finite element spaces $(X_{i}^h,Q_{i}^h)$ are high enough order. Second approach is to define $L_i^H$ on a coarser grid than $(X_{i}^h,Q_{i}^h)$,  see, e.g., \cite{jokaya3,JK08}.  Herein, we will use the first way  which is the most common choice in geophysical problems. Thus, we choose $ L_i^H$ to be piecewise polynomials of degree $k-1$. The choice of the parameter is $\nu_T=h$ is typical for various artificial viscosity-type models.
\end{remark}
With the discrete inf-sup condition, GA-VMS formulation (\ref{BE7})-(\ref{alg4}) can be computed equivalently solving: Find  $(u_{h,1}^{n+1},u_{h,2}^{n+1},\mathbb{G}_1^{\mathbb{H},{n+1}},\mathbb{G}_2^{\mathbb{H},{n+1}}) \in (V_1^h,V_2^h,L_1^H,L_2^H)$ such that
\begin{eqnarray}
\lefteqn{(\frac{{u^{n+1}_{h,1}}-u_{h,1}^{n}}{\Delta t} ,v_{h,1})_{\Omega_1} +(\nu_1+\nu_T)(\nabla  u_{h,1}^{n+1},\nabla v_{h,1})_{\Omega_1}+ c_1(u_{h,1}^{n+1}; u_{h,1}^{n+1},v_{h,1})_{\Omega_1}}
\nonumber\\
&&+\kappa \int_{I}|[\bu_h^n]|u_{h,1}^{n+1}v_{h,1}ds-\kappa \int_{I}u_{h,2}^n|[\bu_h^n]|^{1/2}|[\bu_h^{n-1}]|^{1/2}v_{h,1}ds\nonumber\\
&=&(\bff_1^{n+1},v_{h,1})_{\Omega_1}+\nu_T(\mathbb{G}_1^{\mathbb{H},n},\nabla v_{h,1}), \label{alg1}
\\
\lefteqn{(\mathbb{G}_1^{\mathbb{H},n}-\nabla u_{h,1}^n, \mathbb{L}_1^H)_{\Omega_i}=0, }\label{alg101}
\end{eqnarray}
and
\begin{eqnarray}
\lefteqn{(\frac{{u^{n+1}_{h,2}}-u_{h,2}^{n}}{\Delta t} ,v_{h,2})_{\Omega_2} +(\nu_2+\nu_T)(\nabla  {u}_{h,2}^{n+1},\nabla v_{h,2}^h)_{\Omega_2}
	+ c_2({u}_{h,2}^{n+1};{u}_{h,2}^{n+1},v_{h,2})_{\Omega_2}}
\nonumber\\
&&+\kappa \int_{I}|[\bu_h^n]|u_{h,2}^{n+1}v_{h,2}ds-\kappa \int_{I}u_{h,1}^n|[\bu_h^n]|^{1/2}|[\bu_h^{n-1}]|^{1/2}v_{h,2}ds\nonumber\\
&=&(\bff_2^{n+1},v_{h,2})_{\Omega_2}+\nu_T(\mathbb{G}_2^{\mathbb{H},n},\nabla v_{h,2}), \label{alg2} \\[5pt]
\lefteqn{(\mathbb{G}_2^{\mathbb{H},n}-\nabla u_{h,2}^n, \mathbb{L}_2^H)_{\Omega_i}=0, }\label{alg100}
\end{eqnarray}
for all $(v_{h,1},v_{h,2},\mathbb{L}_1^H,\mathbb{L}_2^H ) \in (V_1^h, V_2^h,L_1^H,L_2^H)$.

\begin{remark}\label{remark:alternative_approach}
Notice that the GA-VMS method (\ref{alg1})-(\ref{alg100}) is derived, based on the variational formulation (\ref{weak1}) - or, equivalently, one could derive (\ref{alg1})-(\ref{alg100}) from \eqref{eq:atmo}-\eqref{eq:atmoBC}, but the coefficients $\nu_i$ would need to be replaced with $\nu_i+\nu_T$ in (\ref{eq:atmoI}). If, however, one tried to create a GA-VMS method from \eqref{eq:atmo}-\eqref{eq:atmoBC}, all the interface integrals in (\ref{alg1})-(\ref{alg100}) would be multiplied by $\frac{\nu_i+\nu_T}{\nu_i}$. Numerical tests show that this alternative approach fails to provide good quality approximations when $\nu_i$ are small.
\end{remark}

\section{Mathematical Preliminaries}

In this section, some inequalities and definitions are introduced. The following lemmas are required for the analysis. %The detailed proof them can be found in \cite{CHL12}.
\begin{lemma} \label{Standard_ineqs} Let $\alpha, \beta, \theta \in H^1(\Omega_i)$  for $i=1,2$,  then there exists constants $ C(\Omega_i)>0$ such that
	\begin{eqnarray}
	c_i({\alpha};{\beta},\theta)_{\Omega_i}&\leq& C (\Omega_i) \|\alpha\|^{1/2}_{\Omega_i}\|\nabla \alpha\|_{\Omega_i}^{1/2}\|\nabla \beta\|_{\Omega_i}\|\nabla \theta\|_{\Omega_i},\nonumber\\
	\int_{I}\alpha |[\beta]| \theta &\leq& C(\Omega_i) \|\alpha \|_I ||[\beta]||_I \|\theta \|_I, \nonumber\\
	\|\alpha\|_I &\leq& C(\Omega_i) \left({\color{red} } \|\alpha\|^{1/4}_{\Omega_i}\|\nabla \alpha\|^{3/4}_{\Omega_i} { + \|\alpha\|^{1/6}_{\Omega_i}\|\nabla \alpha\|^{5/6}_{\Omega_i} }\right).
	\end{eqnarray}
\end{lemma}
\begin{proof}
The first two bounds are standard - see, e.g., Lemma 2.1 on p. 1301 of \cite{CHL12}. The third bound can be found in \cite{Galdi94}, see Theorem II.4.1, p. 63.
\end{proof}
\begin{lemma} \label{lem:nnl}Let $\alpha_i \in X_i$, $\theta_j \in X_j$, $\boldsymbol{\beta} \in H^1(\Omega_i)$ and  $\epsilon_i,\epsilon_j,\varepsilon_i,\varepsilon_j$ ($i,j=1,2$) be positive constants, then one
\begin{eqnarray}
\kappa \int_{I} |\alpha_i| |[ \boldsymbol{\beta}]||\theta_j|&\leq& \frac{C\kappa^2}{4}\|\alpha_i\|^2_I ||[ \boldsymbol{\beta}]||_I^2+\frac{\epsilon_j}{\nu_j^5}\|\theta_j\|^2_{\Omega_j}+\frac{\nu_j}{2\epsilon_j}\|\nabla \theta_j\|^2_{\Omega_j},\\
	\kappa \int_{I} |\alpha_i| |[ \boldsymbol{\beta}]||\theta_j|&\leq& C\kappa^6\Big(\frac{\epsilon_i^5}{\nu_i^5}||[ \boldsymbol{\beta}]||_I^6\|\alpha_i\|^2_{\Omega_i}+\frac{\varepsilon_j^5}{\nu_j^5}||[ \boldsymbol{\beta}]||_I^6\|\theta_j\|^2_{\Omega_j}\Big) \nonumber\\&&+\frac{\nu_i}{4\epsilon_i}\|\nabla \alpha_i\|^2+\frac{\nu_j}{4\varepsilon_j} \|\nabla \theta_j\|^2,\\
	\kappa \int_{I} |\alpha_i| |[ \boldsymbol{\beta}]||\theta_j|&\leq&C\kappa^6\|\alpha_i\|^6_I\Big(\frac{\epsilon_1^5}{\nu_1^5}\|\beta_1\|_{\Omega_1}^2+\frac{\epsilon_2^5}{\nu_2^5}\|\beta_2\|^2_{\Omega_2}+\frac{2\varepsilon_j^5}{\nu_j^5}\|\theta_j\|_{\Omega_j}^2\Big)\nonumber\\&&+\frac{\nu_1}{4\epsilon_1}\|\nabla \beta_1\|_{\Omega_1}^2+\frac{\nu_2}{4\epsilon_2}\|\nabla \beta_2\|^2_{\Omega_2}+\frac{\nu_j}{2\beta_j}\|\nabla \theta_j\|_{\Omega_j}^2.
\end{eqnarray}
\end{lemma}
\begin{proof}
{ Use Lemma \ref{Standard_ineqs} and Young's inequality (see Lemma 2.2 on p. 1302 of \cite{CHL12}).}
\end{proof}
Denoting the corresponding Galerkin approximations of $(u_i,p_i)$ in $(X_i^h,Q_i)$ by $(v_{h,i},q_{h,i})$,  one can assume that the following approximation assumptions (see \cite{GR79}):
\begin{eqnarray}
\inf_{v_{h,i} \in X_i^h} \Big(\|u_i-v_{h,i}\|+\|\nabla(u_i-v_{h,i})\|\Big)&\leq& Ch^{k+1} \|u_i\|_{k+1},\label{inp2}\\\inf_{q_{h,i} \in Q_i^h} \|p_i-q_{h,i}\|&\leq& Ch^{k} \|p_i\|_{k}. \label{inp}
\end{eqnarray}
The  $L^2$ projection is defined in the usual way.
\begin{definition} \label{defpro}
The $L^2$ projection ${P}^H $of a given function $\mathbb{L}$ onto the finite element space $L_i^H$ is the solution of the following : find $\hat{\mathbb{L}}_i= {P}^H \mathbb{L}_i\in L_i^H$ such that
\begin{eqnarray} \label{pro}
(\mathbb{L}_i-{P}^H \mathbb{L}_i, S_H)=0,
\end{eqnarray}
for all $S_H \in L_i^H$.
\end{definition}
Hence, we get
\begin{eqnarray}\label{pro2}
	\|\mathbb{L}_i-{P}^H \mathbb{L}_i\| \leq CH^k\|\mathbb{L}_i\|_{k+1},
\end{eqnarray}
for all $\mathbb{L} \in (L(\Omega_i))^{d\times d}\cap (H^{k+1}(\Omega_i))^{d\times d}$.

We note that while the larger choice of the coarse mesh size $H$ provides more efficient projections into large scale spaces $L_i^H$ and reduces storage, the accuracy of the solutions decreases. For $k=2$, the typical choice is $H=O(h^{1/2})$ for the projection-based VMS. This choice is obtained from balancing terms in the convergence analysis. In our numerical studies, we will use single mesh, that is $H=h$. Although, it is expensive (particularly in $3d$) because of storing the velocity gradient will be the same as storing three additional velocities, it is also good way of programming since there will be less bookkeeping. As we will show later this choice also provides good accuracy.

We also use Poincar\'e-Friedrichs inequality as: There exists a constant $C_p$ such that
\begin{eqnarray}
\|u_{h,i}\|_{\Omega_i}\leq C_p\|\nabla u_{h,i}\|_{\Omega_i},\quad \forall u_{h,i} \in X_i^h
\end{eqnarray}
holds.
Along the paper, we use the following inequality whose proof can be found in \cite{HR86}.
\begin{lemma}\label{gron}[Discrete Gronwall Lemma]
	Let $\gamma_i,\theta_i,\beta_i,\alpha_i$ (for $i\geq 0$), and $\Delta t$, C be a non-negative numbers such that
	\begin{eqnarray*}
		\gamma_M +\Delta t \sum_{i=0}^{M} \theta_i \leq \Delta t \sum_{i=0}^{M} \alpha_i\gamma_i +\Delta t \sum_{i=0}^{M} \beta_i + C,\, \, \forall M\geq 0.
	\end{eqnarray*}
	Assume $\alpha_i\Delta t <1$ for all $i$, then,
	\begin{eqnarray*}
		\gamma_M +\Delta t \sum_{i=0}^{M} \theta_i \leq \exp\Bigg(\Delta t \sum_{i=0}^{M}\theta_i \frac{\alpha_i}{1-\alpha_i\Delta t }\Bigg) \Bigg(\Delta t \sum_{i=0}^{M} \beta_i + C\Bigg),\, \, \forall M\geq 0.
	\end{eqnarray*}
\end{lemma}

%&&+\Delta t\Big(
%2\nu_1\|\nabla  u_{h,1}^{n+1}\|^2_{\Omega_1} + \nu_T\|\nabla  u_{h,1}^{n+1} - \mathbb{G}_1^{\mathbb{H},n}\|^2_{\Omega_1}  \nonumber \\
%&&+\nu_T \|\nabla  u_{h,1}^{n} - \mathbb{G}_1^{\mathbb{H},n}\|^2_{\Omega_1}+\nu_T\big(\|\nabla  u_{h,1}^{n+1}\|^2_{\Omega_1}
%-\|\nabla  u_{h,1}^{n}\|^2_{\Omega_1}\big)
%\Big) \nonumber\\
%&&+\Delta t\Big(
%2\nu_2\|\nabla  u_{h,2}^{n+1}\|^2_{\Omega_2} + \nu_T\|\nabla  u_{h,2}^{n+1} - \mathbb{G}_2^{\mathbb{H},n}\|^2_{\Omega_2}  \nonumber \\
%&&+\nu_T \|\nabla  u_{h,2}^{n} - \mathbb{G}_2^{\mathbb{H},n}\|^2_{\Omega_2}+\nu_T\big(\|\nabla  u_{h,2}^{n+1}\|^2_{\Omega_2}
%-\|\nabla  u_{h,2}^{n}\|^2_{\Omega_2}\big)
%\Big) \nonumber\\

\section{Energy conservation and stability properties of GA-VMS method }

This section considers the energy balance and the stability for the GA-VMS scheme. We first show that the scheme admits an energy balance which is analogous to balances for the continuous AO. Next, we prove its unconditional stability and long-time $L^2$ stability of velocity.
%stability and convergence analysis of \eqref{alg1}-\eqref{alg2}. To this end, we first prove that the system \eqref{alg1}-\eqref{alg2} conserves energy.
\begin{lemma}(Global energy conservation) The scheme \eqref{alg1}-\eqref{alg100} admits the following energy conservation law:
	\begin{eqnarray}
	\lefteqn{\|{u^{M+1}_{h,1}}\|_{\Omega_1}^2 +\|{u^{M+1}_{h,2}}\|_{\Omega_2}^2 + \nu_T\Delta t\big(\|\nabla  u_{h,1}^{M+1}\|^2_{\Omega_1}+\|\nabla  u_{h,2}^{M+1}\|^2_{\Omega_2}\big)} \nonumber \\
	&&+\Delta t \sum_{n=1}^M(\|{{u^{n+1}_{h,1}}-u_{h,1}^{n}}\|_{\Omega_1}^2+\|{{u^{n+1}_{h,2}}-u_{h,2}^{n}}\|_{\Omega_2}^2)\nonumber\\
	&&+\Delta t \sum_{n=1}^M \Big(2\nu_1\|\nabla  u_{h,1}^{n+1}\|^2_{\Omega_1} + \nu_T\|\nabla  u_{h,1}^{n+1} - \mathbb{G}_1^{\mathbb{H},n}\|^2_{\Omega_1}+\nu_T \|\nabla  u_{h,1}^{n} - \mathbb{G}_1^{\mathbb{H},n}\|^2_{\Omega_1} \Big)\nonumber\\
	&&+\Delta t \sum_{n=1}^M \Big(2\nu_2\|\nabla  u_{h,2}^{n+1}\|^2_{\Omega_2} + \nu_T\|\nabla  u_{h,2}^{n+1} - \mathbb{G}_2^{\mathbb{H},n}\|^2_{\Omega_2}+\nu_T \|\nabla  u_{h,2}^{n} - \mathbb{G}_2^{\mathbb{H},n}\|^2_{\Omega_2} \Big)\nonumber\\
	&&+\kappa \Delta t \int_{I}|[\bu_h^M]| (|u_{h,1}^{M+1}|^2+|u_{h,2}^{M+1}|^2)ds\nonumber\\
	&&+\kappa \Delta t \sum_{n=1}^M \int_{I}\Big{|}|[\bu_h^n]|^{1/2}u_{h,1}^{n+1}-|[\bu_h^{n-1}]|^{1/2}u_{h,2}^{n}\Big{|}^2ds\nonumber\\
	&&+\kappa \Delta t \sum_{n=1}^M \int_{I}\Big{|}|[\bu_h^n]|^{1/2}u_{h,2}^{n+1}-|[\bu_h^{n-1}]|^{1/2}u_{h,1}^{n}\Big{|}^2ds \nonumber\allowdisplaybreaks\\
	&=&\|{u^{1}_{h,1}}\|_{\Omega_1}^2+\|{u^{1}_{h,2}}\|_{\Omega_2}^2
	+ \nu_T\Delta t\big(\|\nabla  u_{h,1}^{1}\|^2_{\Omega_1}+\|\nabla  u_{h,2}^{1}\|^2_{\Omega_2}\big)
	+\kappa \Delta t \int_{I}|[u_h^{0}]| (|u_{h,1}^{1}|^2+|u_{h,2}^{1}|^2)ds\nonumber\allowdisplaybreaks\\&&+2\Delta t \sum_{n=1}^M(\bff_1^{n+1},u_{h,1}^{n+1})_{\Omega_1}+2\Delta t \sum_{n=1}^M(\bff_2^{n+1},u_{h,2}^{n+1})_{\Omega_2}.\label{enrgy}
	\end{eqnarray}
\end{lemma}

\begin{proof}
Letting $v_{h,1} =u_{h,1}^{n+1}$ in \eqref{alg1} and $v_{h,2} =u_{h,2}^{n+1}$ in \eqref{alg2} and using the skew-symmetry of nonlinear terms, we get
	\begin{eqnarray}
	\lefteqn{(\frac{{u^{n+1}_{h,1}}-u_{h,1}^{n}}{\Delta t} ,u_{h,1}^{n+1})_{\Omega_1} +(\nu_1+\nu_T)\|\nabla  u_{h,1}^{n+1}\|^2_{\Omega_1}} \nonumber
		\\
		&&+\kappa \int_{I}|[\bu_h^n]||u_{h,1}^{n+1}|^2ds-\kappa \int_{I}u_{h,2}^n|[\bu_h^n]|^{1/2}|[\bu_h^{n-1}]|^{1/2}u_{h,1}^{n+1}ds
		\nonumber\\
	&=&(\bff_1^{n+1},u_{h,1}^{n+1})_{\Omega_1}+\nu_T(\mathbb{G}_1^{\mathbb{H},n},\nabla u_{h,1}^{n+1})_{\Omega_1},\label{stb1}
	\end{eqnarray}
	and
	\begin{eqnarray}
	\lefteqn{(\frac{{u^{n+1}_{h,2}}-u_{h,2}^{n}}{\Delta t} ,u_{h,2}^{n+1})_{\Omega_2} +(\nu_{2}+\nu_T)\|\nabla  {u}_{h,2}^{n+1}\|^2_{\Omega_2}} \nonumber
	\\
&&+\kappa \int_{I}|[\bu_h^n]||u_{h,2}^{n+1}|^2ds-\kappa \int_{I}u_{h,1}^n|[\bu_h^n]|^{1/2}|[\bu_h^{n-1}]|^{1/2}u_{h,2}^{n+1}ds \nonumber\\
	&=&(\bff_2^{n+1},u_{h,2}^{n+1})_{\Omega_2}+\nu_T(\mathbb{G}_2^{\mathbb{H},n},\nabla u_{h,2}^{n+1})_{\Omega_2}.\label{stab2}
	\end{eqnarray}
	Utilizing
	\begin{eqnarray}
	(a-b)\cdot a=\dfrac{1}{2}(|a|^2+|a-b|^2-|b|^2),\label{eq}
	\end{eqnarray}
	we have
	\begin{eqnarray}
	\lefteqn{\dfrac{1}{2\Delta t}(\|{u^{n+1}_{h,1}}\|_{\Omega_1}^2+\|{{u^{n+1}_{h,1}}-u_{h,1}^{n}}\|_{\Omega_1}^2-\|{u^{n}_{h,1}}\|_{\Omega_1}^2)}\nonumber\\
	&&+(\nu_1+\nu_T)\|\nabla  u_{h,1}^{n+1}\|^2_{\Omega_1}-\nu_T(\mathbb{G}_1^{\mathbb{H},n},\nabla u_{h,1}^{n+1})_{\Omega_1}\nonumber\\
	&&+\kappa \int_{I}|[\bu_h^n]| \, \,|u_{h,1}^{n+1}|^2ds -\kappa \int_{I}u_{h,2}^n|[\bu_h^n]|^{1/2}[\bu_h^{n-1}]^{1/2}u_{h,1}^{n+1}ds\nonumber\\
	&=&(\bff_1^{n+1},u_{h,1}^{n+1})_{\Omega_1},\label{stb33}
	\end{eqnarray}
	and
	\begin{eqnarray}
	\lefteqn{\dfrac{1}{2\Delta t}(\|{u^{n+1}_{h,2}}\|_{\Omega_2}^2+\|{{u^{n+1}_{h,2}}-u_{h,2}^{n}}\|_{\Omega_2}^2-\|{u^{n}_{h,2}}\|_{\Omega_2}^2)}\nonumber\\
	&&+(\nu_2+\nu_T)\|\nabla  {u}_{h,2}^{n+1}\|^2_{\Omega_2} - \nu_T(\mathbb{G}_2^{\mathbb{H},n},\nabla u_{h,2}^{n+1})_{\Omega_2}\nonumber\\
	&&+\kappa \int_{I}|[\bu_h^n]||u_{h,2}^{n+1}|^2ds-\kappa \int_{I}u_{h,1}^n|[\bu_h^n]|^{1/2}|[\bu_h^{n-1}]|^{1/2}u_{h,2}^{n+1}ds\nonumber\\
	&=&(\bff_2^{n+1},u_{h,2}^{n+1})_{\Omega_2},\label{stb43}
	\end{eqnarray}
	Adding \eqref{stb33} to \eqref{stb43} and multiplying by $2\Delta t$ yields
	\begin{eqnarray}
	\lefteqn{\|{u^{n+1}_{h,1}}\|_{\Omega_1}^2+\|{{u^{n+1}_{h,1}}-u_{h,1}^{n}}\|_{\Omega_1}^2-\|{u^{n}_{h,1}}\|_{\Omega_1}^2 +\|{u^{n+1}_{h,2}}\|_{\Omega_2}^2+\|{{u^{n+1}_{h,2}}-u_{h,2}^{n}}\|_{\Omega_2}^2-\|{u^{n}_{h,2}}\|_{\Omega_2}^2} \nonumber\\
	&&+2\Delta t\big((\nu_1+\nu_T)\|\nabla  u_{h,1}^{n+1}\|^2_{\Omega_1}-\nu_T (\mathbb{G}_1^{\mathbb{H},n},\nabla u_{h,1}^{n+1})_{\Omega_1}+(\nu_2+\nu_T) \|\nabla  {u}_{h,2}^{n+1}\|_{\Omega_2}-\nu_T (\mathbb{G}_2^{\mathbb{H},n},\nabla u_{h,2}^{n+1})_{\Omega_2}\Big) \nonumber\\
	&&+2\kappa\Delta t \int_{I}|[\bu_h^n]||u_{h,1}^{n+1}|^2ds+2\kappa\Delta t \int_{I}|[\bu_h^n]||u_{h,2}^{n+1}|^2ds \nonumber\\&&-2\kappa\Delta t \int_{I}u_{h,2}^n|[\bu_h^n]|^{1/2}[\bu_h^{n-1}]^{1/2}u_{h,1}^{n+1}ds-2\kappa\Delta t \int_{I}u_{h,1}^n|[\bu_h^n]|^{1/2}|[\bu_h^{n-1}]|^{1/2}u_{h,2}^{n+1}ds\nonumber\\
	&=&2\Delta t(\bff_1^{n+1},u_{h,1}^{n+1})_{\Omega_1}+2\Delta t(\bff_2^{n+1},u_{h,2}^{n+1})_{\Omega_2}.\label{stb5}
	\end{eqnarray}
	The interface terms on the left hand side of \eqref{stb5} can be expressed as (see \cite{CHL12})
	\begin{eqnarray}
	\lefteqn{\kappa \int_{I}|[\bu_h^n]| \, \,|u_{h,1}^{n+1}|^2ds-\kappa \int_{I}u_{h,2}^n|[\bu_h^n]|^{1/2}[\bu_h^{n-1}]^{1/2}u_{h,1}^{n+1}ds}\nonumber\\
	&&+\kappa \int_{I}|[\bu_h^n]||u_{h,2}^{n+1}|^2ds-\kappa \int_{I}u_{h,1}^n|[\bu_h^n]|^{1/2}|[\bu_h^{n-1}]|^{1/2}u_{h,2}^{n+1}ds\nonumber\\
	&=&\dfrac{\kappa}{2} \int_{I}|[\bu_h^n]| (|u_{h,1}^{n+1}|^2+|u_{h,2}^{n+1}|^2)ds- \dfrac{\kappa}{2}\int_{I}|[\bu_h^{n-1}]| (|u_{h,1}^{n}|^2+|u_{h,2}^{n}|^2)ds\nonumber\\
	&&+\dfrac{\kappa}{2} \int_{I}\Big{|}|[\bu_h^n]|^{1/2}u_{h,1}^{n+1}-|[\bu_h^{n-1}]|^{1/2}u_{h,2}^{n}\Big{|}^2ds\nonumber\\
	&&+\dfrac{\kappa}{2} \int_{I}\Big{|}|[\bu_h^n]|^{1/2}u_{h,2}^{n+1}-|[\bu_h^{n-1}]|^{1/2}u_{h,1}^{n}\Big{|}^2ds. \label{stb6}
	\end{eqnarray}
	Substituting \eqref{stb6} into \eqref{stb5} gives
	\begin{eqnarray}
	\lefteqn{\|{u^{n+1}_{h,1}}\|_{\Omega_1}^2+\|{{u^{n+1}_{h,1}}-u_{h,1}^{n}}\|_{\Omega_1}^2-\|{u^{n}_{h,1}}\|_{\Omega_1}^2 +\|{u^{n+1}_{h,2}}\|_{\Omega_2}^2+\|{{u^{n+1}_{h,2}}-u_{h,2}^{n}}\|_{\Omega_2}^2-\|{u^{n}_{h,2}}\|_{\Omega_2}^2}\nonumber\\
	&&+2\Delta t\big((\nu_1+\nu_T)\|\nabla  u_{h,1}^{n+1}\|^2_{\Omega_1}-\nu_T (\mathbb{G}_1^{\mathbb{H},n},\nabla u_{h,1}^{n+1})_{\Omega_1}+(\nu_2+\nu_T) \|\nabla  {u}_{h,2}^{n+1}\|_{\Omega_2}-\nu_T (\mathbb{G}_2^{\mathbb{H},n},\nabla u_{h,2}^{n+1})_{\Omega_2}\Big) \nonumber\\
	&&+\kappa \Delta t \int_{I}|[\bu_h^n]| (|u_{h,1}^{n+1}|^2+|u_{h,2}^{n+1}|^2)ds
	- \kappa \Delta t\int_{I}|[\bu_h^{n-1}]| (|u_{h,1}^{n}|^2+|u_{h,2}^{n}|^2)ds\nonumber\\
	&&+\kappa \Delta t \int_{I}\Big{|}|[\bu_h^n]|^{1/2}u_{h,1}^{n+1}-|[\bu_h^{n-1}]|^{1/2}u_{h,2}^{n}\Big{|}^2ds
	+\kappa \Delta t \int_{I}\Big{|}|[\bu_h^n]|^{1/2}u_{h,2}^{n+1}-|[\bu_h^{n-1}]|^{1/2}u_{h,1}^{n}\Big{|}^2ds \nonumber\\
	&=&2\Delta t(\bff_1^{n+1},u_{h,1}^{n+1})_{\Omega_1}+2\Delta t(\bff_2^{n+1},u_{h,2}^{n+1})_{\Omega_2},\label{stb7}
	\end{eqnarray}
	    Also considering the fact that $(\nabla u_{h,i}^n - \mathbb{G}_i^{\mathbb{H},n}, \mathbb{G}_i^{\mathbb{H},n})_{\Omega_i}=0,$ one can easily show
\begin{equation}
\|\nabla u_{h,i}^{n} - \mathbb{G}_i^{\mathbb{H},n}\|^2_{\Omega_i}=
\|\nabla u_{h,i}^{n}\|^2_{\Omega_i}-\|\mathbb{G}_i^{\mathbb{H},n}\|^2_{\Omega_i}. \nonumber
\end{equation}
The last equality and some algebraic manipulations give
\begin{eqnarray}
(\nu_i+\nu_T)\|\nabla  u_{h,i}^{n+1}\|^2_{\Omega_i} - \nu_T(\mathbb{G}_i^{\mathbb{H},n},\nabla u_{h,i}^{n+1})_{\Omega_i} \nonumber\\
=
\nu_i\|\nabla  u_{h,i}^{n+1}\|^2_{\Omega_i} + \frac{\nu_T}{2}\big(\|\nabla  u_{h,i}^{n+1} - \mathbb{G}_i^{\mathbb{H},n}\|^2_{\Omega_i} + 2(\mathbb{G}_i^{\mathbb{H},n},\nabla u_{h,i}^{n+1})_{\Omega_i} - \|\mathbb{G}_i^{\mathbb{H},n}\|^2_{\Omega_i} \big) \nonumber \\
-\nu_T (\mathbb{G}_i^{\mathbb{H},n},\nabla u_{h,i}^{n+1})_{\Omega_i} +\frac{\nu_T}{2}\big(\|\nabla  u_{h,i}^{n+1}\|^2_{\Omega_i} - \|\nabla  u_{h,i}^{n}\|^2_{\Omega_i}\big) +\frac{\nu_T}{2} \|\nabla  u_{h,i}^{n}\|^2_{\Omega_i} \nonumber \\
=
\nu_i\|\nabla  u_{h,i}^{n+1}\|^2_{\Omega_i} + \frac{\nu_T}{2}\|\nabla  u_{h,i}^{n+1} - \mathbb{G}_i^{\mathbb{H},n}\|^2_{\Omega_i}
+\frac{\nu_T}{2} \|\nabla  u_{h,i}^{n} - \mathbb{G}_i^{\mathbb{H},n}\|^2_{\Omega_i} \nonumber \\
+\frac{\nu_T}{2}\big(\|\nabla  u_{h,i}^{n+1}\|^2_{\Omega_i}
-\|\nabla  u_{h,i}^{n}\|^2_{\Omega_i}\big).
\nonumber
\end{eqnarray}
	Substituting the last equation in (\ref{stb7}),
\begin{eqnarray}
	\lefteqn{\|{u^{n+1}_{h,1}}\|_{\Omega_1}^2+\|{{u^{n+1}_{h,1}}-u_{h,1}^{n}}\|_{\Omega_1}^2-\|{u^{n}_{h,1}}\|_{\Omega_1}^2 +\|{u^{n+1}_{h,2}}\|_{\Omega_2}^2+\|{{u^{n+1}_{h,2}}-u_{h,2}^{n}}\|_{\Omega_2}^2-\|{u^{n}_{h,2}}\|_{\Omega_2}^2}\nonumber\\
&&+\Delta t\Big(
2\nu_1\|\nabla  u_{h,1}^{n+1}\|^2_{\Omega_1} + \nu_T\|\nabla  u_{h,1}^{n+1} - \mathbb{G}_1^{\mathbb{H},n}\|^2_{\Omega_1}  \nonumber \\
&&+\nu_T \|\nabla  u_{h,1}^{n} - \mathbb{G}_1^{\mathbb{H},n}\|^2_{\Omega_1}+\nu_T\big(\|\nabla  u_{h,1}^{n+1}\|^2_{\Omega_1}
-\|\nabla  u_{h,1}^{n}\|^2_{\Omega_1}\big)
\Big) \nonumber\\
&&+\Delta t\Big(
2\nu_2\|\nabla  u_{h,2}^{n+1}\|^2_{\Omega_2} + \nu_T\|\nabla  u_{h,2}^{n+1} - \mathbb{G}_2^{\mathbb{H},n}\|^2_{\Omega_2}  \nonumber \\
&&+\nu_T \|\nabla  u_{h,2}^{n} - \mathbb{G}_2^{\mathbb{H},n}\|^2_{\Omega_2}+\nu_T\big(\|\nabla  u_{h,2}^{n+1}\|^2_{\Omega_2}
-\|\nabla  u_{h,2}^{n}\|^2_{\Omega_2}\big)
\Big) \nonumber\\
	&&+\kappa \Delta t \int_{I}|[\bu_h^n]| (|u_{h,1}^{n+1}|^2+|u_{h,2}^{n+1}|^2)ds
	- \kappa \Delta t\int_{I}|[\bu_h^{n-1}]| (|u_{h,1}^{n}|^2+|u_{h,2}^{n}|^2)ds\nonumber\\
	&&+\kappa \Delta t \int_{I}\Big{|}|[\bu_h^n]|^{1/2}u_{h,1}^{n+1}-|[\bu_h^{n-1}]|^{1/2}u_{h,2}^{n}\Big{|}^2ds
	+\kappa \Delta t \int_{I}\Big{|}|[\bu_h^n]|^{1/2}u_{h,2}^{n+1}-|[\bu_h^{n-1}]|^{1/2}u_{h,1}^{n}\Big{|}^2ds \nonumber\\
	&=&2\Delta t(\bff_1^{n+1},u_{h,1}^{n+1})_{\Omega_1}+2\Delta t(\bff_2^{n+1},u_{h,2}^{n+1})_{\Omega_2},\label{stb8}
	\end{eqnarray}
	Summing over the time levels completes the proof.
\end{proof}

We now provide  the stability of  \eqref{alg1}-\eqref{alg100}.
%The stability analysis of (\ref{alg1})-(\eqref{alg100}) is given below.
\begin{lemma} \label{l:stab}
	Let $f_i\in L^2(0,T; H^{-1}(\Omega_i))$ for $i=1,2$. {The} scheme (\ref{alg1})-\eqref{alg100} is unconditionally stable and provides the following bound at time step $t=M+1$
	\begin{eqnarray}
	\lefteqn{\|{u^{M+1}_{h,1}}\|_{\Omega_1}^2 +\|{u^{M+1}_{h,2}}\|_{\Omega_2}^2+\nu_T\Delta t( \|\nabla  u_{h,1}^{M+1}\|^2_{\Omega_1}+\|\nabla  {u}_{h,2}^{M+1}\|^2_{\Omega_2})}\nonumber\\
	&&+\nu_1\Delta t \sum_{n=1}^M\|\nabla  u_{h,1}^{n+1}\|^2_{\Omega_1}+\nu_2\Delta t\sum_{n=1}^M\|\nabla  {u}_{h,2}^{n+1}\|^2_{\Omega_2}\nonumber\\&&+\kappa \Delta t \int_{I}|[\bu_h^M]| (|u_{h,1}^{M+1}|^2+|u_{h,2}^{M+1}|^2)ds\nonumber\\
	&&+\kappa \Delta t \sum_{n=1}^M \int_{I}\Big{|}|[\bu_h^n]|^{1/2}u_{h,1}^{n+1}-|[\bu_h^{n-1}]|^{1/2}u_{h,2}^{n}\Big{|}^2ds\nonumber\\
	&&+\kappa \Delta t \sum_{n=1}^M \int_{I}\Big{|}|[\bu_h^n]|^{1/2}u_{h,2}^{n+1}-|[\bu_h^{n-1}]|^{1/2}u_{h,1}^{n}\Big{|}^2ds \nonumber\\
	&\leq&\|{u^{1}_{h,1}}\|_{\Omega_1}^2+\|{u^{1}_{h,2}}\|_{\Omega_2}^2+\kappa \Delta t \int_{I}|[u^{0}]| (|u_{h,1}^{1}|^2+|u_{h,2}^{1}|^2)ds+\nu_T\Delta t( \|\nabla u_{h,1}^{1}\|_{\Omega_1}^2 +\|\nabla u_{h,2}^{1}\|_{\Omega_2}^2 )\nonumber\\&&+{\Delta t }\sum_{n=1}^M(\nu_1^{-1}\|f_1^{n+1}\|_{-1,\Omega_1}^2+\nu_2^{-1}\|f_2^{n+1}\|_{-1,\Omega_2}^2)\label{stb}
	\end{eqnarray}
\end{lemma}
\begin{proof}
	Performing Cauchy-Schwarz and Young's inequalities for the right side of energy conservation equation (\ref{enrgy}), we have
	\begin{eqnarray}
	2\Delta t\sum_{n=1}^M(f_1^{n+1},u_{h,1}^{n+1})_{\Omega_1}&\leq& {\nu_1^{-1} \Delta t}\sum_{n=1}^M\|f_1^{n+1}\|_{-1,\Omega_1}^2+{\nu_1\Delta t}\sum_{n=1}^M\|\nabla u_{h,1}^{n+1}\|_{\Omega_1}^2, \label{f1}\\2\Delta t \sum_{n=1}^M(\bff_2^{n+1},u_{h,2}^{n+1})_{\Omega_2}&\leq&{\nu_2^{-1} \Delta t}\sum_{n=1}^M\|f_2^{n+1}\|_{-1,\Omega_2}^2+{\nu_2\Delta t}\sum_{n=1}^M\|\nabla u_{h,2}^{n+1}\|_{\Omega_2}^2.\label{f2}
	\end{eqnarray}
	Letting $\mathbb{L}_i^H=\mathbb{G}_i^{\mathbb{H},n}$ in \eqref{alg101} and \eqref{alg100} and utilizing Cauchy-Schwarz inequality gives
	%\begin{eqnarray}
	%	\|\mathbb{G}_i^{\mathbb{H},n}\|^2&\leq &\|\nabla u_i^{n}\|\|\mathbb{G}_i^{\mathbb{H},n}\| \label{g3}
	%\end{eqnarray}
	%which produces
	\begin{eqnarray}
	\|\mathbb{G}_i^{\mathbb{H},n}\|_{\Omega_i}&\leq &\|\nabla u_{h,i}^{n}\|_{\Omega_i}. \label{g4}
	\end{eqnarray}
	Thus, using \eqref{g4}, the last two terms on the right hand side of \eqref{enrgy} can be bounded as
	\begin{eqnarray}
	2\nu_T\Delta t\sum_{n=1}^M(\mathbb{G}_1^{\mathbb{H},n},\nabla u_{h,1}^{n+1})_{\Omega_1}
	&\leq& {\nu_T\Delta t}\sum_{n=1}^M\|\nabla u_{h,1}^n\|_{\Omega_1}^2+{\nu_T\Delta t}\sum_{n=1}^M\|\nabla u_{h,1}^{n+1}\|_{\Omega_1}^2, \label{g1}\\
	[7pt]2\nu_T\Delta t\sum_{n=1}^M(\mathbb{G}_2^{\mathbb{H},n},\nabla u_{h,2}^{n+1})_{\Omega_2}
	&\leq& {\nu_T \Delta t}\sum_{n=1}^M\|\nabla u_{h,2}^n\|_{\Omega_2}^2+{\nu_T\Delta t}\sum_{n=1}^M\|\nabla u_{h,2}^{n+1}\|_{\Omega_2}^2. \label{g2}
	\end{eqnarray}
	Substituting (\ref{f1})-\eqref{f2} and \eqref{g1}-(\ref{g2}) in (\ref{enrgy}) produces the required result.
%	\begin{eqnarray}
	%\lefteqn{\|{u^{M+1}_{h,1}}\|_{\Omega_1}^2 +\|{u^{M+1}_{h,2}}\|_{\Omega_2}^2+(\nu_1+\nu_T)\Delta t \sum_{n=1}^M\|\nabla  u_{h,1}^{n+1}\|_{\Omega_1}}\nonumber\\
	%&&+(\nu_2+\nu_T)\Delta t\sum_{n=1}^M\|\nabla  {u}_{h,2}^{n+1}\|_{\Omega_2}+\kappa \Delta t \int_{I}|[\bu_h^M]| (|u_{h,1}^{M+1}|^2+|u_{h,2}^{M+1}|^2)ds\nonumber\\
	%&&+\kappa \Delta t \sum_{n=1}^M \int_{I}\Big{|}|[\bu_h^n]|^{1/2}u_{h,1}^{n+1}-|[\bu_h^{n-1}]|^{1/2}u_{h,2}^{n}\Big{|}^2ds\nonumber\\
	%&&+\kappa \Delta t \sum_{n=1}^M \int_{I}\Big{|}|[\bu_h^n]|^{1/2}u_{h,2}^{n+1}-|[\bu_h^{n-1}]|^{1/2}u_{h,1}^{n}\Big{|}^2ds \nonumber\\
	%&\leq&\|{u^{0}_{h,1}}\|_{\Omega_1}^2+\|{u^{0}_{h,2}}\|_{\Omega_2}^2+\kappa \Delta t \int_{I}|[u_h^{0}]| %(|u_{h,1}^{1}|^2+|u_{h,2}^{1}|^2)ds\nonumber\\&&\Delta t\sum_{n=1}^M(\nu_1^{-1}\|f_1^{n+1}\|_{-1}^2+\nu_2^{-1}\|f_2^{n+1}\|_{-1}^2)+\nu_T\Delta t\sum_{n=1}^M( \|\nabla u_{h,1}^{n}\|_{\Omega_1}^2 +\|\nabla u_{h,2}^{n}\|_{\Omega_2}^2 ).\label{stab3}
%	\end{eqnarray}
	%Simplifying the terms in \eqref{stab3} yields \eqref{stb}.
\end{proof}
%\begin{remark}
%Note that the unconditional stabilities of $G_H^n$ and
%$\Phi_H^n$ are guaranteed through \eqref{staU}-\eqref{staB}
%and the result of Lemma \ref{l:stab}.
%\end{remark}
%\begin{remark}
%Since the scheme is linear,  at each timestep the steps of the stability proof can be repeated to show
%uniqueness of solutions.  Also, since the scheme is finite dimensional as well, we have existence and uniqueness
%of solutions at each timestep, and therefore for the entire scheme.
%\end{remark}
%{\color{blue}Songul says: This is also new. We showed the long time stability property.Please check.

We next prove that (\ref{alg1})-\eqref{alg100} is unconditionally long-time stable. To perform the long-time stability, in view of Lemma \ref{l:stab}, the right-hand side of \eqref{stb} is denoted by $S_M$, 
\begin{gather}
S_M:=\|{u^{1}_{h,1}}\|_{\Omega_1}^2+\|{u^{1}_{h,2}}\|_{\Omega_2}^2+\kappa \Delta t \int_{I}|[u^{0}]| (|u_{h,1}^{1}|^2+|u_{h,2}^{1}|^2)ds+\nu_T\Delta t( \|\nabla u_{h,1}^{1}\|_{\Omega_1}^2 +\|\nabla u_{h,2}^{1}\|_{\Omega_2}^2 ) \nonumber
\\
+{\Delta t }\sum_{n=1}^M(\nu_1^{-1}\|f_1^{n+1}\|_{-1,\Omega_1}^2+\nu_2^{-1}\|f_2^{n+1}\|_{-1,\Omega_2}^2). 
\label{stab}
\end{gather}

\begin{lemma} \label{lll1}Let $f_i \in L^{\infty}(\mathbb{R}^+,V_i^{h*})$ for $i=1,2$ be given, then solutions of the scheme (\ref{alg1})-\eqref{alg100} are long-time stable in the following sense: for any time step $\Delta t >0$ and for any $n> 0$
		\begin{eqnarray}
	\lefteqn{\|{u^{n+1}_{h,1}}\|_{\Omega_1}^2 +\nu_T\Delta t\|\nabla  u_{h,1}^{N+1}\|^2_{\Omega_1}+{\kappa} \Delta t \int_{I}|[\bu_h^n]| (|u_{h,1}^{n+1}|^2)ds}\nonumber\\&&+\|{u^{n+1}_{h,2}}\|_{\Omega_2}^2 +\nu_T\Delta t\|\nabla  u_{h,2}^{n+1}\|^2_{\Omega_2} +{\kappa} \Delta t \int_{I}|[\bu_h^n]| (|u_{h,2}^{n+1}|^2)ds\nonumber\\
	&\leq&(1+\alpha)^{-n}\Big(\|{u^{1}_{h,1}}\|_{\Omega_1}^2+\nu_T\Delta t \|\nabla u_{h,1}^{1}\|_{\Omega_1}^2+\kappa \Delta t\int_{I}|[\bu_h^{0}]| (|u_{h,1}^{1}|^2)ds\nonumber\\
	&&+\|{u^{1}_{h,2}}\|_{\Omega_2}^2+\nu_T\Delta t \|\nabla u_{h,2}^{1}\|_{\Omega_2}^2+\kappa \Delta t\int_{I}|[\bu_h^{0}]| (|u_{h,2}^{1}|^2)ds\Big)\nonumber
	\\&& +\alpha^{-1}\Delta t(\nu_1^{-1}\|f_1\|_{L^{\infty}(\mathbb{R}^+,V_1^{h*})}^2+\nu_2^{-1}\|f_2\|_{L^{\infty}(\mathbb{R}^+,V_2^{h*})}^2),\label{lstb}
	\end{eqnarray}
	where $\alpha:=\min\Big\{\dfrac{\nu_i\Delta t}{3C_p^2},\dfrac{\nu_i}{3\nu_T},\dfrac{\nu_i}{3}(\dfrac{C\kappa^2S_n}{4}+\dfrac{2C_p^2}{\nu_i^5}+\dfrac{\nu_i}{4})^{-1}\Big\}$, for $i=1,2$.
\end{lemma}
\begin{proof} Adding \eqref{stb33} to \eqref{stb43}, applying Cauchy-Schwarz and Young's inequalities, using \eqref{stb6}, \eqref{g4}, and dropping the non-negative terms, we have
	\begin{eqnarray}
	\lefteqn{\|{u^{n+1}_{h,1}}\|_{\Omega_1}^2 +\nu_T\Delta t\|\nabla  u_{h,1}^{n+1}\|^2_{\Omega_1}+{\kappa} \Delta t \int_{I}|[\bu_h^n]| (|u_{h,1}^{n+1}|^2)ds}\nonumber\\&&+\|{u^{n+1}_{h,2}}\|_{\Omega_2}^2 +\nu_T\Delta t\|\nabla  u_{h,2}^{n+1}\|^2_{\Omega_2} +{\kappa} \Delta t \int_{I}|[\bu_h^n]| (|u_{h,2}^{n+1}|^2)ds\nonumber\\
	&&+{\nu_1}\Delta t\|\nabla  u_{h,1}^{n+1}\|^2_{\Omega_1}+{\nu_2}\Delta t\|\nabla  u_{h,2}^{n+1}\|^2_{\Omega_2}\nonumber\\
	&\leq&\|{u^{n}_{h,1}}\|_{\Omega_1}^2+\nu_T\Delta t \|\nabla u_{h,1}^{n}\|_{\Omega_1}^2+\kappa \Delta t\int_{I}|[\bu_h^{n-1}]| (|u_{h,1}^{n}|^2)ds \nonumber\\
	&&+\|{u^{n}_{h,2}}\|_{\Omega_2}^2+\nu_T\Delta t \|\nabla u_{h,2}^{n}\|_{\Omega_2}^2+\kappa \Delta t\int_{I}|[\bu_h^{n-1}]| (|u_{h,2}^{n}|^2)ds\nonumber\\
	&&+\Delta t\nu_1^{-1}\|f_1^{n+1}\|_{V_1^{h*}}^2+\Delta t\nu_2^{-1}\|f_2^{n+1}\|_{V_1^{h*}}^2.\label{lstb1}
	\end{eqnarray}
Using the Lemma \ref{lem:nnl} with $\varepsilon=2$, Poincar\'e inequality and Lemma \ref{l:stab} produce
	\begin{eqnarray}
\kappa	\int_{I}|[\bu_h^n]| |u_{h,i}^{n+1}|^2\nonumber&\leq&\dfrac{C\kappa^2}{4} |[\bu_h^n]|_I^2 |u_{h,i}^{n+1}|_I^2+(\dfrac{2C_p^2}{\nu_i^5}+\dfrac{\nu_i}{4})\|\nabla u_{h,i}^{n+1}\|^2\nonumber\\
	&\leq&\dfrac{C\kappa^2}{4} |[\bu_h^n]|_I^2 \|u_{h,i}^{n+1}\|_{\Omega_i}^{1/3}\|\nabla u_{h,i}^{n+1}\|_{\Omega_i}^{5/3}+(\dfrac{2C_p^2}{\nu_i^5}+\dfrac{\nu_i}{4})\|\nabla u_{h,1}^{n+1}\|_{\Omega_i}^2\nonumber\\
		&\leq&(\dfrac{C\kappa^2S_n}{4}+\dfrac{2C_p^2}{\nu_i^5}+\dfrac{\nu_i}{4})\|\nabla u_{h,i}^{n+1}\|_{\Omega_i}^2, \quad \text{for} \, \, \, i=1,2, \label{lstb2}
	\end{eqnarray}
where $S_n$ has been defined in \eqref{stab}. Thus, the last two terms on the left hand side of \eqref{lstb1} can be written as
	\begin{eqnarray}
\lefteqn{{\nu_1}\Delta t\|\nabla  u_{h,1}^{n+1}\|^2_{\Omega_1}+{\nu_2}\Delta t\|\nabla  u_{h,2}^{n+1}\|^2_{\Omega_2}}\nonumber\\%&\geq& \dfrac{\nu_1}{3C_p^2}\Delta t\| u_{h,1}^{n+1}\|^2_{\Omega_1}+\dfrac{\nu_1}{6\nu_T}2\nu_T\Delta t\|\nabla  u_{h,1}^{n+1}\|^2_{\Omega_1}+\dfrac{\nu_1}{3}\Delta t\|\nabla  u_{h,1}^{n+1}\|^2_{\Omega_1}\nonumber\\
	%	&\geq& \dfrac{\nu_1}{3C_p^2}\Delta t\| u_{h,1}^{n+1}\|^2_{\Omega_1}+\dfrac{\nu_1}{6\nu_T}2\nu_T\Delta t\|\nabla  u_{h,1}^{n+1}\|^2_{\Omega_1}+\dfrac{\nu_1}{3}(S+\dfrac{4C_p^2}{\nu_1^5}+\dfrac{\nu_1}{2})^{-1}\Delta t({2\kappa} \int_{I}|[\bu_h^n]| |u_{h,1}^{n+1}|^2-|[\bu_h^n]|^2 |u_{h,1}^{n+1}|^2)\nonumber\\
		 &\geq&\alpha\Big(\|{u^{n+1}_{h,1}}\|_{\Omega_1}^2 +\nu_T\Delta t\|\nabla  u_{h,1}^{n+1}\|^2_{\Omega_1} +{\kappa} \Delta t \int_{I}|[\bu_h^n]| (|u_{h,1}^{n+1}|^2)ds\nonumber\\&&+\|{u^{n+1}_{h,2}}\|_{\Omega_2}^2 +\nu_T\Delta t\|\nabla  u_{h,2}^{n+1}\|^2_{\Omega_2} +{\kappa} \Delta t \int_{I}|[\bu_h^n]| (|u_{h,2}^{n+1}|^2)ds\Big),\label{lstb3}
	\end{eqnarray}
	where $\alpha:=\min\Big\{\dfrac{\nu_i\Delta t}{3C_p^2},\dfrac{\nu_i}{3\nu_T},\dfrac{\nu_i}{3}(\dfrac{C\kappa^2S_n}{4}+\dfrac{2C_p^2}{\nu_i^5}+\dfrac{\nu_i}{4})^{-1}\Big\}$, for $i=1,2$. Inserting \eqref{lstb3} in \eqref{lstb1} and multiplying by $(1+\alpha)^{-1}$, we obtain
	\begin{eqnarray}
\lefteqn{\|{u^{n+1}_{h,1}}\|_{\Omega_1}^2 +\nu_T\Delta t\|\nabla  u_{h,1}^{n+1}\|^2_{\Omega_1}+{\kappa} \Delta t \int_{I}|[\bu_h^n]| (|u_{h,1}^{n+1}|^2)ds}\nonumber\\&&+\|{u^{n+1}_{h,2}}\|_{\Omega_2}^2 +\nu_T\Delta t\|\nabla  u_{h,2}^{n+1}\|^2_{\Omega_2} +{\kappa} \Delta t \int_{I}|[\bu_h^n]| (|u_{h,2}^{n+1}|^2)ds\nonumber\\
&\leq&(1+\alpha)^{-1}\Big(\|{u^{n}_{h,1}}\|_{\Omega_1}^2+\nu_T\Delta t \|\nabla u_{h,1}^{n}\|_{\Omega_1}^2+\kappa \Delta t\int_{I}|[\bu_h^{n-1}]| (|u_{h,1}^{n}|^2)ds\nonumber\\
&&+\|{u^{n}_{h,2}}\|_{\Omega_2}^2+\nu_T\Delta t \|\nabla u_{h,2}^{n}\|_{\Omega_2}^2+\kappa \Delta t\int_{I}|[\bu_h^{n-1}]| (|u_{h,2}^{n}|^2)ds\Big)\nonumber
\\&& +\Delta t\nu_1^{-1}(1+\alpha)^{-1}\|f_1^{n+1}\|_{V_1^{h*}}^2+\Delta t\nu_2^{-1}(1+\alpha)^{-1}\|f_2^{n+1}\|_{V_1^{h*}}^2.\label{lstb4}
\end{eqnarray}
Utilizing induction produces the stated result \eqref{lstb}.
	\begin{remark}
Lemma \ref{lll1} proves that the long-time velocity solutions are bounded by the problem data and it is independent of the initial conditions when n is sufficiently large.
\end{remark}	
\end{proof}
%}
\section{Convergence Analysis }
This section presents convergence analysis of (\ref{alg1})-(\ref{alg2}). It is assumed that all functions are sufficiently regular, i.e. the solution of \eqref{eq:atmo}-\eqref{eq:atmoBC} satisfies
\begin{eqnarray}
u \in{L^\infty(0,T;H^{k+1}(\Omega) \cap H^3(\Omega) )}, \quad
\partial_t u \in{L^\infty(0,T;H^{k+1}(\Omega)^d)}\label{reg},\quad \partial_{tt} u\in{L^\infty(0,T;H^{1}(\Omega)^d)} \label{r}.
\end{eqnarray}
%The mesh and velocity approximating polynomial degree $k$ is chosen so that the Taylor-Hood element is inf-sup stable and the properties \eqref{app1}-\eqref{app2} hold.
We need to define the following discrete norms to use in the convergence analysis.
\begin{eqnarray}
|||u|||_{\infty,p}=\max\limits_{0\leq j \leq N}||u(t^j)||_p, \, \, |||u|||_{s,p}=\Big(\Delta t \sum\limits_{j=1}^{M}||u(t^j)||_p^s \Big)^{\frac{1}{s}}.
\end{eqnarray}

 %{\color{blue} Alex is this $E$ depend only $u_h$ }

{Following the notation of \cite{CHL12}, let $D^{n+1} = \tilde{\nu}^5(1+\kappa^6E^{n+1}+|||\nabla u|||_{\infty,\Omega}^4)$, where $\tilde{\nu} = \max\{\nu^{-1}_1,\nu^{-1}_2\}$ and $E^{n+1} = \max_{j=0,1,...,n+1}\{\max\{\|u(t^j)\|^6_I,\|u_h^j\|^6_I\}\}$.}

\begin{theorem}\label{The:conv} { Let the time step be chosen so that $\Delta t \le 1/D^{n+1}$. Then the following bound on the error holds under the regularity assumptions \eqref{reg}:}
	\begin{eqnarray}
	\lefteqn{\|{\bf u}(t^{{M}+1})-{\bf u}^{{M}+1}\|^2  +{ \frac{3}{4}(\nu_1+\nu_T)\Delta t\sum_{n=1}^{M}\|\nabla  (u_1(t^{n+1})-u_{h,1}^{n+1})\|^2}}\nonumber\\&&+2\kappa\Delta t \sum_{n=1}^{M}\int_{I}|[{\bf u}^n]|  |{\bf u}(t^{n+1})-{\bf u}^{n+1}|^2ds +{ \frac{3}{4}(\nu_2+\nu_T)\Delta t\sum_{n=1}^{M}\|\nabla  (u_2(t^{n+1})-u_{h,2}^{n+1})\|^2}
	\nonumber\\&\leq& \|{\bf u}(t^{1})-{\bf u}_h^{1}\|^2+\frac{(\nu_1{ +\nu_T})\Delta t}{8}(2\|\nabla({ u_1}(t^{1})-u_{h,1}^{1})\|^2_{\Omega_1}+\|\nabla ({ u_1}(t^{0})-u_{h,1}^{0})\|^2_{\Omega_1})\nonumber\\[3pt]&&+\frac{(\nu_2{ +\nu_T})\Delta t}{8}(2\|\nabla({ u_{2}}(t^{1})-{u_{h,2}^{1}})\|^2_{\Omega_2}+\|\nabla ({ u_2}(t^{0})-u_{h,2}^{0})\|^2_{\Omega_2})\nonumber\\&&+C(\Delta t^2+h^{2k}+H^{2k}),\label{thm}
	\end{eqnarray}
	where $C$ is a generic constant depending only on ${f_i},{\nu_i+\nu_T},\Omega$.
\end{theorem}

\begin{proof}
The finite element error analysis starts by deriving error equations for GA-VMS finite element method (\ref{alg1})-(\ref{alg2}) by subtracting the scheme from weak formulation of \eqref{eq:atmo}-\eqref{eq:atmoBC} . To do this, first note that the true solution of \eqref{eq:atmo}-\eqref{eq:atmoBC} at time $t^{n+1}$ satisfies
	\begin{eqnarray}
	\lefteqn{(\frac{{u_1(t^{n+1})}-u_{1}(t^{n})}{\Delta t} ,v_{h,1})_{\Omega_1} +(\nu_1+\nu_T)(\nabla  u_{1}(t^{n+1}),\nabla v_{h,1})_{\Omega_1}-(p_1(t^{n+1}),\nabla\cdot v_{h,1})_{\Omega_1}}
	\nonumber\\
	&&+\kappa \int_{I}(u_1(t^{n+1})-u_{2}(t^{n+1}))|[{\bf u}(t^{n+1})]|v_{h,1}ds+ {c_1(u_{1}(t^{n+1});u_{1}(t^{n+1}),v_{h,1})}\nonumber\allowdisplaybreaks\\&=&(\frac{{u_1(t^{n+1})}-u_{1}(t^{n})}{\Delta t}-\partial_{t}u_1(t^{n+1}) ,v_{h,1})_{\Omega_1}+\nu_T(\nabla  u_{1}(t^{n+1}),\nabla v_{h,1})_{\Omega_1}\nonumber\\&&+(\bff_1^{n+1},v_{h,1})_{\Omega_1} \label{err1}
	\end{eqnarray}
	and
	\begin{eqnarray}
	\lefteqn{(\frac{{u_{2}(t^{n+1})}-u_{2}(t^{n})}{\Delta t} ,v_{h,2})_{\Omega_2} +(\nu_2+\nu_T)(\nabla  {u}_{2}(t^{n+1}),\nabla v_{h,2})_{\Omega_2}-(p_2(t^{n+1}),\nabla\cdot v_{h,2})_{\Omega_2}}
	\nonumber\allowdisplaybreaks\\
	&&+\kappa \int_{I}(u_{2}(t^{n+1})-u_1(t^{n+1}))|[{\bf u}(t^{n+1})]|v_{h,2}ds+ {c_2(u_{2}(t^{n+1});u_{2}(t^{n+1}),v_{h,2})}\nonumber\allowdisplaybreaks\\&=&(\frac{{u_{2}(t^{n+1})}-u_{2}(t^{n})}{\Delta t} -\partial_{t}u_2(t^{n+1}),v_{h,2})_{\Omega_2}+\nu_T(\nabla  {u}_{2}(t^{n+1}),\nabla v_{h,2})_{\Omega_2}\nonumber\\&&+(\bff_2^{n+1},v_{h,2})_{\Omega_2}, \label{err2}
	\end{eqnarray}
	for all $(v_{h,1}, v_{h,2}) \in (V_1^h,V_2^h)$. For arbitrary $\tilde{u}_1^{n+1}\in V_1^h$ and $\tilde{u}_2^{n+1}\in V_2^h$, the error is decomposed into
	\begin{eqnarray}
	e_1^{n+1}&=&u_1(t^{n+1})-u_{h,1}^{n+1}=(u_1(t^{n+1})-\tilde{u_1}^{n+1})-(u_{h,1}^{n+1}-\tilde{u_1}^{n+1})=:\eta_1^{n+1}-\phi_{h,1}^{n+1},\nonumber\\
	e_2^{n+1}&=&u_{2}(t^{n+1})-u_{h,2}^{n+1}=(u_{2}(t^{n+1})-\tilde{u_{2}}^{n+1})-(u_{h,2}^{n+1}-\tilde{u_{2}}^{n+1}=:\eta_2^{n+1}-\phi_{h,2}^{n+1}.
	\end{eqnarray}
The interpolation error can be estimated with \eqref{inp2}. Thus, subtracting \eqref{alg1}-\eqref{alg100} from \eqref{err1}-\eqref{err2} gives
	\begin{eqnarray}
	\lefteqn{(\frac{\phi_{h,1}^{n+1}-\phi_{h,1}^{n}}{\Delta t} ,v_{h,1})_{\Omega_1} +(\nu_1+\nu_T)(\nabla  \phi_{h,1}^{n+1},\nabla v_{h,1})_{\Omega_1}+\kappa \int_{I}|[{\bf u}_h^n]|u_{h,1}^{n+1}v_{h,1}ds}
		\nonumber\\
		&&-\kappa \int_{I}|[{\bf u}(t^{n+1})]|u_1(t^{n+1})v_{h,1}ds+\kappa \int_{I}u_{2}(t^{n+1})|[{\bf u}(t^{n+1})]|v_{h,1}ds\nonumber
		\\&&-\kappa \int_{I}u_{h,2}^n|[{\bf u}_h^n]|^{1/2}|[{\bf u}_h^{n-1}]|^{1/2}v_{h,1}ds
	\nonumber\\&=&(\frac{\eta_{1}^{n+1}-\eta_{1}^{n}}{\Delta t} ,v_{h,1})_{\Omega_1}+(\nu_1+\nu_T)(\nabla  \eta_{1}^{n+1},\nabla v_{h,1})_{\Omega_1}-(p_1(t^{n+1})-q_1^{n+1},\nabla\cdot v_{h,1})_{\Omega_1}\nonumber\\&&+(\partial_{t}u_1(t^{n+1})-\frac{{u_1(t^{n+1})}-u_{1}(t^{n})}{\Delta t} ,v_{h,1})_{\Omega_1}+\nu_T(\mathbb{G}_1^{\mathbb{H},n}-\nabla u_1(t^{n+1}),\nabla v_{h,1})\nonumber\\&&+{c_1(u_{1}(t^{n+1});u_{1}(t^{n+1}),v_{h,1})}-{c_1(u_{h,1}^{n+1};u_{h,1}^{n+1},v_{h,1})}, \label{err3}
	\end{eqnarray}
	and
	\begin{eqnarray}
	\lefteqn{(\frac{\phi_{h,2}^{n+1}-\phi_{h,2}^{n}}{\Delta t} ,v_{h,2})_{\Omega_2} +(\nu_2+\nu_T)(\nabla  \phi_{h,2}^{n+1},\nabla v_{h,2})_{\Omega_2}+\kappa \int_{I}|[{\bf u}_h^n]|u_{h,2}^{n+1}v_{h,2}ds}
		\nonumber\\
		&&-\kappa \int_{I}|[{\bf u}(t^{n+1})]|u_{2}(t^{n+1})v_{h,2}ds+\kappa \int_{I}u_1(t^{n+1})|[{\bf u}(t^{n+1})]|v_{h,2}ds\nonumber\\&&-\kappa \int_{I}u_{h,1}^n|[{\bf u}_h^n]|^{1/2}|[{\bf u}_h^{n-1}]|^{1/2}v_{h,2}ds\nonumber
	\nonumber\\&=&(\frac{\eta_{2}^{n+1}-\eta_{2}^{n}}{\Delta t} ,v_{h,2})_{\Omega_2}+(\nu_2+\nu_T)(\nabla  \eta_{2}^{n+1},\nabla v_{h,2})_{\Omega_2}-(p_2(t^{n+1})-q_2^{n+1},\nabla\cdot v_{h,2})_{\Omega_2}\nonumber\\&&+(\partial_{t}u_2(t^{n+1})-\frac{{u_{2}(t^{n+1})}-u_{2}(t^{n})}{\Delta t} ,v_{h,2})_{\Omega_2}+\nu_T(\mathbb{G}_2^{\mathbb{H},n}-\nabla u_{2}(t^{n+1}),\nabla v_{h,2})\nonumber\\&&+{c_2(u_{2}(t^{n+1});u_{2}(t^{n+1}),v_{h,2})}-{c_2(u_{h,2}^{n+1};u_{h,2}^{n+1},v_{h,2})}. \label{err4}
	\end{eqnarray}
	Then choosing $v_{h,1}=\phi_{h,1}^{n+1}$ in \eqref{err3} and using the polarization identity \eqref{eq} provides
	\begin{eqnarray}
	\lefteqn{\frac{1}{2\Delta t}\Big(\|\phi_{h,1}^{n+1}\|^2_{\Omega_1}-\|\phi_{h,1}^{n}\|^2_{\Omega_1}+\|\phi_{h,1}^{n+1}-\phi_{h,1}^{n}\|^2_{\Omega_1}\Big) +(\nu_1+\nu_T)\|\nabla  \phi_{h,1}^{n+1}\|_{\Omega_1}}
		\nonumber\\
		&&+\kappa \int_{I}|[{\bf u}_h^n]|u_{h,1}^{n+1}\phi_{h,1}^{n+1}ds- \kappa\int_{I}|[{\bf u}(t^{n+1})]|u_1(t^{n+1})\phi_{h,1}^{n+1}ds\nonumber\\&&+\kappa  \int_{I}u_{2}(t^{n+1})|[{\bf u}(t^{n+1})]|\phi_{h,1}^{n+1}ds- \kappa\int_{I}u_{h,2}^n|[{\bf u}_h^n]|^{1/2}|[{\bf u}_h^{n-1}]|^{1/2}\phi_{h,1}^{n+1}ds
	\nonumber\\&\leq&\abs{(\frac{\eta_{1}^{n+1}-\eta_{1}^{n}}{\Delta t} ,\phi_{h,1}^{n+1})_{\Omega_1}}+(\nu_1+\nu_T)\abs{(\nabla  \eta_{1}^{n+1},\nabla \phi_{h,1}^{n+1})_{\Omega_1}}+\abs{(p_1(t^{n+1})-q_1^{n+1},\nabla\cdot \phi_{h,1}^{n+1})_{\Omega_1}}\nonumber\\&&+\abs{(\partial_{t}u_1(t^{n+1})-\frac{{u_1(t^{n+1})}-u_{1}(t^{n})}{\Delta t} ,\phi_{h,1}^{n+1})_{\Omega_1}}+\nu_T\abs{(\mathbb{G}_1^{\mathbb{H},n}-\nabla u_1(t^{n+1}),\nabla \phi_{h,1}^{n+1})}\nonumber\\&&+{c_1(u_{1}(t^{n+1});u_{1}(t^{n+1}),\phi_{h,1}^{n+1})}-{c_1(u_{h,1}^{n+1};u_{h,1}^{n+1},\phi_{h,1}^{n+1})}. \label{er3}
	\end{eqnarray}
	Applying Cauchy-Schwarz, Young's, and Poincar\'e inequalities along with Taylor theorem, we get
	\begin{eqnarray}
	\abs{(\frac{\eta_{1}^{n+1}-\eta_{1}^{n}}{\Delta t} ,\phi_{h,1}^{n+1})_{\Omega_1}}&\leq&C{ (\nu_1+\nu_T)}^{-1}{\Delta t}^{-1}\int_{t^n}^{t^{n+1}}\|\partial_t\eta_{1}^{n+1}\|_{\Omega_1}^2\nonumber\\&&+\frac{{ (\nu_1+\nu_T)}}{36}\|\nabla \phi_{h,1}^{n+1}\|_{\Omega_1}^2, \label{err0}\\
	(\nu_1+\nu_T)\abs{(\nabla  \eta_{1}^{n+1},\nabla \phi_{h,1}^{n+1})_{\Omega_1}}&\leq&C{ (\nu_1+\nu_T)}\|\nabla \eta_1^{n+1}\|_{\Omega_1}^2\nonumber\\&&+\frac{{ (\nu_1+\nu_T)}}{36}\|\nabla \phi_{h,1}^{n+1}\|_{\Omega_1}^2,\\
	\abs{(p_1(t^{n+1})-q_1^{n+1},\nabla\cdot \phi_{h,1}^{n+1})_{\Omega_1}}&\leq&C{ (\nu_1+\nu_T)}^{-1}\|p_1(t^{n+1})-q_1^{n+1}\|_{\Omega_1}^2\nonumber\\&&+\frac{{ (\nu_1+\nu_T)}}{36}\|\nabla \phi_{h,1}^{n+1}\|_{\Omega_1}^2,\\
	\abs{(\partial_{t}u_1(t^{n+1})-\frac{{u_1(t^{n+1})}-u_{1}(t^{n})}{\Delta t} ,\phi_{h,1}^{n+1})_{\Omega_1}}&\leq&C{ (\nu_1+\nu_T)}^{-1}\|\partial_t u_1(t^{n+1})-\frac{{u_1(t^{n+1})}-u_{1}(t^{n})}{\Delta t}\|_{\Omega_1}^2\nonumber\\&&+\frac{{ (\nu_1+\nu_T)}}{36}\|\nabla \phi_{h,1}^{n+1}\|_{\Omega_1}^2.%\\\nu_T\abs{(\mathbb{G}_1^{\mathbb{H},n}-\nabla u_1(t^{n+1}),\nabla \phi_{h,1}^{n+1})_{{\Omega_1}}}&\leq&C\nu_T^2(\nu_1+\nu_T)^{-1}{\color{red} \|\mathbb{G}_1^{\mathbb{H},n}-\nabla u_1(t^{n+1})\|_{\Omega_1}^2}\nonumber\\&&+{\color{red}\frac{\nu_1+\nu_T}{4}\|\nabla \phi_{h,1}^{n+1}\|_{\Omega_1}^2},
	\label{err5}
	\end{eqnarray}
		The equations (\ref{alg101}) and (\ref{alg100}) state that $\mathbb{G}_i^{\mathbb{H},n}=P^H \nabla u_{h,i}^n$ where $P^H$ is the $L^2(\Omega_i)$-orthogonal projection defined by (\ref{pro}). Hence, utilizing Cauchy–Schwarz and Young's inequality to the fifth term on the right hand side of \eqref{er3} yields % Utilizing , Hence The .. term can be bounded as
	\begin{eqnarray}
	\lefteqn{\nu_T\abs{(\mathbb{G}_1^{\mathbb{H},n}-\nabla u_1(t^{n+1}),\nabla \phi_{h,1}^{n+1})_{{\Omega_1}}}}\nonumber\\&\leq&(P^H \nabla (u_{h,1}^n-u_{1}(t^n)),\nabla \phi_{h,1}^{n+1})_{{\Omega_1}}-((I-P^H )\nabla u_{1}(t^n),\nabla \phi_{h,1}^{n+1})_{{\Omega_1}}\nonumber\\&&-(\nabla ( u_1(t^{n+1})-u_{1}(t^n)),\nabla \phi_{h,1}^{n+1})_{{\Omega_1}}\nonumber\\&\leq&C\nu_T^2(\nu_1+\nu_T)^{-1}\Big(\|P^H \nabla\eta_1^n\|_{\Omega_1}^2+\|P^H\nabla\phi_{h,1}^n\|_{\Omega_1}^2\nonumber\\&&+\|(I-P^H )\nabla u_{1}(t^n)\|_{\Omega_1}^2+\|\nabla ( u_1(t^{n+1})-u_{1}(t^n))\|_{\Omega_1}^2\Big)\nonumber\\
	&&+{\frac{\nu_1+\nu_T}{36}\|\nabla \phi_{h,1}^{n+1}\|_{\Omega_1}^2}. \label{er1}
	\end{eqnarray}
Taylor remainder formula is used along with \eqref{pro}, \eqref{pro2} and inverse inequality to get
		\begin{eqnarray}
\lefteqn{	\nu_T\abs{(\mathbb{G}_1^{\mathbb{H},n}-\nabla u_1(t^{n+1}),\nabla \phi_{h,1}^{n+1})_{{\Omega_1}}}}\nonumber\\&\leq&C\nu_T^2(\nu_1+\nu_T)^{-1}\Big(\|\nabla \eta_1^n\|^2+h^{-2}\|\phi_{h,1}^n\|^2+H^{2k}\|u_1(t^n)\|_{k+1}^2\nonumber\\&&+\Delta t^2\|\partial_t u_1\|_{L^{\infty}(t^n,t^{n+1};H^1(\Omega))}^2\Big) +{\frac{\nu_1+\nu_T}{36}\|\nabla \phi_{h,1}^{n+1}\|_{\Omega_1}^2}.\label{er2}
	\end{eqnarray}
	The nonlinear terms can be rearranged by adding and subtracting terms and using \\${c_1(u_{h,1}^{n+1};\phi_{h,1}^{n+1},\phi_{h,1}^{n+1})}=0$ as follows.
	\begin{eqnarray}
	\lefteqn{ {c_1(u_{1}(t^{n+1});u_{1}(t^{n+1}),\phi_{h,1}^{n+1})}-{c_1(u_{h,1}^{n+1};u_{h,1}^{n+1},\phi_{h,1}^{n+1})}}\nonumber\\&=&{c_1(\eta_{1}^{n+1};u_{1}(t^{n+1}),\phi_{h,1}^{n+1})}-{c_1(\phi_{h,1}^{n+1};u_{1}(t^{n+1}),\phi_{h,1}^{n+1})}\nonumber\\&&+{c_1(u_{h,1}^{n+1};\eta_{1}^{n+1},\phi_{h,1}^{n+1})}. \label{nnl}
	\end{eqnarray}
	Bounds for the terms on the right hand side of \eqref{nnl} are given as
	\begin{eqnarray}
		{c_1(\eta_{1}^{n+1};u_{1}(t^{n+1}),\phi_{h,1}^{n+1})}&\leq&C{ (\nu_1+\nu_T)}^{-1}\|\nabla\eta_{1}^{n+1}\|_{\Omega_1}^2\|\nabla u_{1}(t^{n+1})\|_{\Omega_1}^2\nonumber\\&&+\frac{{ (\nu_1+\nu_T)}}{36}\|\nabla \phi_{h,1}^{n+1}\|_{\Omega_1}^2,\nonumber\\[5pt]
	{c_1(\phi_{h,1}^{n+1};u_{1}(t^{n+1}),\phi_{h,1}^{n+1})}&\leq&C\|\phi_{h,1}^{n+1}\|_{\Omega_1}^{1/2}\|\nabla \phi_{h,1}^{n+1}\|_{\Omega_1}^{1/2}\|\nabla u_{1}(t^{n+1})\|_{\Omega_1}\|\nabla\phi_{h,1}^{n+1}\|_{\Omega_1}\nonumber\\
	&\leq&C{ (\nu_1+\nu_T)}^{-3}\|\phi_{h,1}^{n+1}\|_{\Omega_1}^2\|\nabla u_{1}(t^{n+1})\|_{\Omega_1}^4\nonumber\\&&+\frac{{ (\nu_1+\nu_T)}}{36}\|\nabla \phi_{h,1}^{n+1}\|_{\Omega_1}^2,\nonumber\\[5pt]
	{c_1(u_{h,1}^{n+1};\eta_{1}^{n+1},\phi_{h,1}^{n+1})} &\leq& C{ (\nu_1+\nu_T)}^{-1}\|\nabla\eta_{1}^{n+1}\|_{\Omega_1}^2\|\nabla u_{h,1}^{n+1}\|_{\Omega_1}^2\nonumber\\&&+\frac{{ (\nu_1+\nu_T)}}{36}\|\nabla \phi_{h,1}^{n+1}\|_{\Omega_1}^2. \label{err7}
	\end{eqnarray}
	The interface integrals can be expressed as  %$\kappa \int_{I}|[{\bf u}^n]|\tilde{u_1}^{n+1}\phi_{h,1}^{n+1}ds$
	\begin{eqnarray}
	\lefteqn{\kappa \int_{I}|[{\bf u}_h^n]|u_{h,1}^{n+1}\phi_{h,1}^{n+1}ds-\kappa \int_{I}|[{\bf u}(t^{n+1})]|u_1(t^{n+1})\phi_{h,1}^{n+1}ds}\nonumber\\&=&-\kappa \int_{I}|[{\bf u}_h^n]|  |\phi_{h,1}^{n+1}|^2ds+\kappa \int_{I}|[{\bf u}_h^{n}]|\eta_{1}^{n+1}\phi_{h,1}^{n+1}ds\nonumber\\&&+\kappa \int_{I}(|[{\bf u}_h^n]|-|[\tilde{\bf u}^n]|)u_1(t^{n+1})\phi_{h,1}^{n+1}ds\nonumber\\&&
	+\kappa \int_{I}(|[\tilde{\bf u}^n]|-|[{\bf u}(t^n)]|)u_1(t^{n+1})\phi_{h,1}^{n+1}ds\nonumber\\
	&&+\kappa \int_{I}(|[{\bf u}(t^n)]|-|[{\bf u}(t^{n+1})]|)u_1(t^{n+1})\phi_{h,1}^{n+1}ds,\label{err8}
	\end{eqnarray}
	and
	\begin{eqnarray}
	\lefteqn{\kappa \int_{I}u_{2}(t^{n+1})|[{\bf u}(t^{n+1})]|\phi_{h,1}^{n+1}ds-\kappa \int_{I}u_{h,2}^n|[{\bf u}^n]|^{1/2}|[{\bf u}^{n-1}]|^{1/2}\phi_{h,1}^{n+1}ds}\nonumber\\&=&\kappa \int_{I}(u_{2}(t^{n+1})-u_{2}(t^{n}))|[{\bf u}(t^{n+1})]|\phi_{h,1}^{n+1}ds+\kappa \int_{I}(\phi_{h,2}^{n}-\eta_{2}^n)|[{\bf u}(t^{n+1})]|\phi_{h,1}^{n+1}ds\nonumber\\&&+\kappa \int_{I}u_{h,2}^n\Big(|[{\bf u}(t^{n+1})]|-\frac{1}{2}(|[{\bf u}(t^{n})]|+|[{\bf u}(t^{n-1})]|)\Big)\phi_{h,1}^{n+1}ds\nonumber\\
	&&+\kappa \int_{I}u_{h,2}^n\Big(\frac{1}{2}(|[{\bf u}(t^{n})]|+|[{\bf u}(t^{n-1})]|-\frac{1}{2}(|[\tilde{\bf u}(t^{n})]|+|[\tilde{\bf u}(t^{n-1})]|))\Big)\phi_{h,1}^{n+1}ds\nonumber\\
	&&+\kappa \int_{I}u_{h,2}^n\Big(\frac{1}{2}(|[\tilde{\bf u}(t^{n})]|+|[\tilde{\bf u}(t^{n-1})]|)-\frac{1}{2}(|[{\bf u}_h^{n}]|+|[{\bf u}_h^{n-1}]|)\Big)\phi_{h,1}^{n+1}ds\nonumber\\
	&&+\kappa \int_{I}u_{h,2}^n\Big(\frac{1}{2}(|[{\bf u}_h^{n}]|+|[{\bf u}_h^{n-1}]|)-|[{\bf u}_h^n]|^{1/2}|[{\bf u}_h^{n-1}]|^{1/2}\Big)\phi_{h,1}^{n+1}ds. \label{err9}
	\end{eqnarray}
	With the use of Lemma \ref{lem:nnl} and the following inequalities
	\begin{eqnarray}
	\Big| |[{\bf u}(t^n)]|-|[\tilde{\bf u}^n]|\Big|&\leq&|[\boldsymbol{ \eta}^n]|,\nonumber\\
	\Big| |[{\bf u}_h^n]|-|[\tilde{\bf u}^n]|\Big|&\leq&|[\boldsymbol{ \phi}_h^n]|,
	\end{eqnarray}
	we bound the terms on the right hand side of \eqref{err8} as
	\begin{eqnarray}
	\lefteqn{\kappa \int_{I}|\eta_{1}^{n+1}||[{\bf u}_h^{n}]||\phi_{h,1}^{n+1}|ds}\nonumber\\&\leq&\frac{C\kappa^2}{4}\|\eta_{1}^{n+1}\|_I^2||[{\bf u}_h^{n}]||^2_I+C{{ (\nu_1+\nu_T)^{-5}}}\|\phi_{h,1}^{n+1}\|_{\Omega_1}^2\allowdisplaybreaks+\frac{{ (\nu_1+\nu_T)}}{36}\|\nabla \phi_{h,1}^{n+1}\|_{\Omega_1}^2,\label{err11}\\[7pt]
	\lefteqn{\kappa \int_{I}|u_1(t^{n+1})|(|[{\bf u}_h^n]|-|[\tilde{\bf u}^n]|)|\phi_{h,1}^{n+1}|ds}\nonumber \\%&\leq& \kappa  \int_{I}u_1(t^{n+1})|[{\bf\phi}_h^n]|\phi_{h,1}^{n+1}ds\nonumber\\
	&\leq&C\kappa^{6}\|u_1(t^{n+1})\|^{6}_I\Big({{ (\nu_1+\nu_T)^{-5}}}\|\phi_{h,1}^n\|^2_{\Omega_1}+{ (\nu_2+\nu_T)^{-5}}\|\phi_{h,2}^n\|^2_{\Omega_2}\nonumber\allowdisplaybreaks\\&&+{ (\nu_1+\nu_T)^{-5}}\|\phi_{h,1}^{n+1}\|^2_{\Omega_1}\Big)+\frac{{ (\nu_1+\nu_T)}}{32}\|\nabla \phi_{h,1}^n\|^2_{\Omega_1}\nonumber\\&&+\frac{{ (\nu_2+\nu_T)}}{48}\|\nabla \phi_{h,2}^n\|^2_{\Omega_2}+\frac{{ (\nu_1+\nu_T)}}{36}\|\nabla \phi_{h,1}^{n+1}\|^2_{\Omega_1},\label{err12}\allowdisplaybreaks\\[7pt]
\lefteqn{	\kappa \int_{I}|u_1(t^{n+1})|(|[\tilde{\bf u}^n]|-|[{\bf u}(t^n)]|)|\phi_{h,1}^{n+1}|ds}\nonumber\\&\leq&\frac{C\kappa^2}{4}\|u_1(t^{n+1})\|_I^2||[{\boldsymbol{\eta}}^n]||^2_I+C{{ (\nu_1+\nu_T)^{-5}}}\|\phi_{h,1}^{n+1}\|_{\Omega_1}^2\nonumber\\&&+\frac{{ (\nu_1+\nu_T)}}{36}\|\nabla \phi_{h,1}^{n+1}\|_{\Omega_1}^2,\label{err13}\allowdisplaybreaks\\[7pt]
\lefteqn{	\kappa \int_{I}|u_1(t^{n+1})||[{\bf u}(t^n)-{\bf u}(t^{n+1})]||\phi_{h,1}^{n+1}|ds}\nonumber\\&\leq&\frac{C\kappa^2}{4}\|u_1(t^{n+1})\|_I^2||[{\bf u}(t^n)-{\bf u}(t^{n+1})]||^2_I+C{{ (\nu_1+\nu_T)^{-5}}}\|\phi_{h,1}^{n+1}\|_{\Omega_1}^2\nonumber\\&&+\frac{{ (\nu_1+\nu_T)}}{36}\|\nabla \phi_{h,1}^{n+1}\|_{\Omega_1}^2 \label{err14}.
	\end{eqnarray}
	Similarly, the first six terms on the right hand side of \eqref{err9} become
	\begin{eqnarray}
	\lefteqn{\kappa \int_{I}|u_{2}(t^{n+1})-u_{2}(t^{n})||[{\bf u}(t^{n+1})]||\phi_{h,1}^{n+1}|ds}\nonumber\\&\leq&\frac{C\kappa^2}{4}\|u_{2}(t^{n+1})-u_{2}(t^{n})\|_I^2||[{\bf u}(t^{n+1})]||^2_I+C{{ (\nu_1+\nu_T)^{-5}}}\|\phi_{h,1}^{n+1}\|_{\Omega_1}^2\nonumber\\&&+\frac{{ (\nu_1+\nu_T)}}{36}\|\nabla \phi_{h,1}^{n+1}\|_{\Omega_1}^2\label{err15},\allowdisplaybreaks\\[7pt]
	\lefteqn{\kappa \int_{I}|\phi_{h,2}^{n}||[{\bf u}(t^{n+1})]||\phi_{h,1}^{n+1}|ds}\nonumber\\&\leq& C\kappa^{6}||[{\bf u}(t^{n+1})]||_I^{6}\Big({{ (\nu_2+\nu_T)^{-5}}}\|\phi_{h,2}^{n}\|^2_{\Omega_2}+{{ (\nu_1+\nu_T)^{-5}}}\|\phi_{h,1}^{n+1}\|^2_{\Omega_1}\Big)\nonumber\allowdisplaybreaks\\&&+\frac{{ (\nu_2+\nu_T)}}{48}\|\nabla \phi_{h,2}^{n}\|^2_{\Omega_2}+\frac{{ (\nu_1+\nu_T)}}{36}\|\nabla \phi_{h,1}^{n+1}\|^2_{\Omega_1},
	\label{err16}\\[7pt]
	\lefteqn{\kappa \int_{I}|\eta_{2}^n||[{\bf u}(t^{n+1})]||\phi_{h,1}^{n+1}|ds\leq\frac{C\kappa^2}{4}\|\eta_{2}^n\|_I^2||[{\bf u}(t^{n+1})]||^2_I+C{{ (\nu_1+\nu_T)^{-5}}}\|\phi_{h,1}^{n+1}\|_{\Omega_1}^2}\nonumber\\&&+\frac{{ (\nu_1+\nu_T)}}{36}\|\nabla \phi_{h,1}^{n+1}\|_{\Omega_1}^2,\label{err17}\\[7pt]
	\lefteqn{\kappa \int_{I}|u_{h,2}^n|\Big(|[{\bf u}(t^{n+1})]|-\frac{1}{2}(|[{\bf u}(t^{n})]|+|[{\bf u}(t^{n-1})]|)\Big)|\phi_{h,1}^{n+1}|ds}\nonumber\\
	&\leq&\frac{\kappa}{2}\int_{I}|u_{h,2}^n|\Big(|[{\bf u}(t^{n})-{\bf u}(t^{n+1})]|+|[{\bf u}(t^{n-1})-{\bf u}(t^{n+1})]|\Big)|\phi_{h,1}^{n+1}|ds\nonumber\\&\leq&\frac{C\kappa^2}{8}\|u_{h,2}^n\|_I^2\Big(|[{\bf u}(t^{n})-{\bf u}(t^{n+1})]|_I^2+|[{\bf u}(t^{n-1})-{\bf u}(t^{n+1})]|_I^2\Big)\nonumber\\&&+C{{ (\nu_1+\nu_T)^{-5}}}\|\phi_{h,1}^{n+1}\|_{\Omega_1}^2+\frac{{ (\nu_1+\nu_T)}}{36}\|\nabla \phi_{h,1}^{n+1}\|_{\Omega_1}^2,\label{err18}\allowdisplaybreaks\\[7pt]
	\lefteqn{\dfrac{\kappa}{2} \int_{I}|u_{h,2}^n|\Big(|[{\bf u}(t^{n})]|+|[{\bf u}(t^{n-1})]|-|[\tilde{\bf u}(t^{n})]|-|[\tilde{\bf u}(t^{n-1})]|\Big)|\phi_{h,1}^{n+1}|ds}\nonumber\\&\leq&\frac{C\kappa^2}{4}\|u_{h,2}^n\|_I^2(||[\boldsymbol{ \eta}^n]||^2_I+||[\boldsymbol{\eta}^{n-1}]||^2_I)+C{{ (\nu_1+\nu_T)^{-5}}}\|\phi_{h,1}^{n+1}\|_{\Omega_1}^2\nonumber\\&&+\frac{{ (\nu_1+\nu_T)}}{36}\|\nabla \phi_{h,1}^{n+1}\|_{\Omega_1}^2,\label{err19}\\[7pt]
	\lefteqn{\dfrac{\kappa}{2} \int_{I}|u_{h,2}^n|\Big(|[\tilde{\bf u}(t^{n})]|+|[\tilde{\bf u}(t^{n-1})]|-|[{\bf u}_h^{n}]|-|[{\bf u}_h^{n-1}]|\Big)|\phi_{h,1}^{n+1}|ds}\nonumber\\&\leq&C\kappa^{6}\|u_{h,2}^n\|^{6}_I\Big({{ (\nu_1+\nu_T)^{-5}}}\|{\bf \phi}_{h,1}^n\|^2_{\Omega_1}+{ (\nu_2+\nu_T)^{-5}}\|{\bf \phi}_{h,2}^n\|^2_{\Omega_2}+{ (\nu_1+\nu_T)^{-5}}\|{\bf \phi}_{h,1}^{n-1}\|^2_{\Omega_1}\nonumber\\&&+{ (\nu_2+\nu_T)^{-5}}\|{\bf \phi}_{h,2}^{n-1}\|^2_{\Omega_2}+{ (\nu_1+\nu_T)^{-5}}\|\phi_{h,1}^{n+1}\|^2_{\Omega_1}\Big)+\Big(\frac{{ (\nu_1+\nu_T)}}{32}\|\nabla{\bf \phi}_{h,1}^n\|^2_{\Omega_1}\nonumber\\&&+\frac{{ (\nu_2+\nu_T)}}{48}\|\nabla{\bf \phi}_{h,2}^n\|^2_{\Omega_2}+\frac{{ (\nu_1+\nu_T)}}{16}\|\nabla {\bf \phi}_{h,1}^{n-1}\|^2_{\Omega_1}+\frac{{ (\nu_2+\nu_T)}}{16}\|\nabla {\bf \phi}_{h,2}^{n-1}\|^2_{\Omega_2}\nonumber\\&&+\frac{{ (\nu_1+\nu_T)}}{36}\|\nabla \phi_{h,1}^{n+1}\|^2_{\Omega_1}\Big). \label{err20}
	\end{eqnarray}
	The last term of \eqref{err9} can be written as
	\begin{eqnarray}
	\lefteqn{\kappa \int_{I}|u_{h,2}^n|\Big(\frac{1}{2}(|[{\bf u}_h^{n}]|+|[{\bf u}_h^{n-1}]|)-|[{\bf u}_h^n]|^{1/2}|[{\bf u}_h^{n-1}]|^{1/2}\Big)|\phi_{h,1}^{n+1}|ds}\nonumber\\&\leq&\frac{\kappa}{2} \int_{I}|u_{h,2}^n|\Big(|[\boldsymbol{\phi}_h^n]|+|[\boldsymbol{\phi}_h^{n-1}]|+|[\boldsymbol{\eta}^n]|+|[\boldsymbol{\eta}^{n-1}]|+|[{\bf u}(t^n)-{\bf u}(t^{n-1})]|\Big)|\phi_{h,1}^{n+1}|ds, \label{err21}
	\end{eqnarray}
	%The terms on the right hand side of \eqref{err21}
	which can be bounded in a similar way to \eqref{err19} and \eqref{err20}. Inserting all bounds 	in \eqref{er3} and multiplying by $2\Delta t$ %\eqref{err0}-\eqref{err20}
gives
	\begin{eqnarray}
	\lefteqn{{\|\phi_{h,1}^{n+1}\|_{\Omega_1}^2-\|\phi_{h,1}^{n}\|_{\Omega_1}^2} +\|\phi_{h,1}^{n+1}-\phi_{h,1}^{n}\|_{\Omega_1}^2 +(\nu_1+\nu_T)\Delta t\|\nabla  \phi_{h,1}^{n+1}\|_{\Omega_1}}
	\nonumber\\&&+2\kappa\Delta t \int_{I}|[{\bf u}_h^n]|  |\phi_{h,1}^{n+1}|^2ds-\frac{{ \nu_1+\nu_T}}{8}\Delta t(\|\nabla{\bf \phi}_{h,1}^n\|^2_{\Omega_1}+\|\nabla {\bf \phi}_{h,1}^{n-1}\|^2_{\Omega_1})\nonumber\allowdisplaybreaks\\&&-\frac{{ \nu_2+\nu_T}}{8}\Delta t(\|\nabla{\bf \phi}_{h,2}^n\|^2_{\Omega_2}+\|\nabla {\bf \phi}_{h,2}^{n-1}\|^2_{\Omega_2})\nonumber\allowdisplaybreaks\\&\leq& C\Delta t\Bigg(\Delta t^{-1}\int_{t^n}^{t^{n+1}}\|\partial_t\eta_{1}^{n+1}\|_{\Omega_1}^2+\Big(1+\|\nabla u_{1}(t^{n+1})\|_{\Omega_1}^2+\|\nabla u_{h,1}^{n+1}\|_{\Omega_1}^2\Big)\|\nabla \eta_1^{n+1}\|^2\nonumber\\&&+\|p_1(t^{n+1})-q_1^{n+1}\|_{\Omega_1}^2+\|\partial_t u_1(t^{n+1})-\frac{{u_1(t^{n+1})}-u_{1}(t^{n})}{\Delta t}\|_{\Omega_1}^2+\|\nabla \eta_1^n\|_{\Omega_1}^2\nonumber\\&&+H^{2k}\|u_1(t^n)\|_{k+1}^2+\Delta t^2\|\partial_t u_1\|_{L^{\infty}(t^n,t^{n+1};H^1(\Omega))}^2+\|\eta_{1}^{n+1}\|_I^2||[{\bf u}_h^{n}]||^2_I+\|u_1(t^{n+1})\|_I^2||[\boldsymbol{\eta}^n]||^2_I\nonumber\\[5pt]&&
	+\|u_1(t^{n+1})\|_I^2||[{\bf u}(t^n)-{\bf u}(t^{n+1})]||^2_I+\|u_{2}(t^{n+1})-u_{2}(t^{n})\|_I^2||[{\bf u}(t^{n+1})]||^2_I\nonumber\\[5pt]
	&&+\|\eta_{2}^n\|_I^2||[{\bf u}(t^{n+1})]||^2_I+\|u_{h,2}^n\|_I^2(||[\boldsymbol{ \eta}^n]||^2_I+||[\boldsymbol{\eta}^{n-1}]||^2_I)\nonumber\\
	&&+\|u_{h,2}^n\|_I^2\Big(|[{\bf u}(t^{n})-{\bf u}(t^{n+1})]|_I^2+|[{\bf u}(t^{n-1})-{\bf u}(t^{n+1})]|_I^2+|[{\bf u}(t^{n})-{\bf u}(t^{n-1})]|_I^2\Big)\nonumber\allowdisplaybreaks\\
	&&+\bigg(\|\nabla u_{1}(t^{n+1})\|_{\Omega_1}^{4}+\kappa^{6}(\|u_1(t^{n+1})\|^{6}_I+||[{\bf u}(t^{n+1})]||^{6}_I+\|u_{h,2}^n\|^{6}_I)\bigg)\|\phi_{h,1}^{n+1}\|_{\Omega_1}^2\nonumber\\&&+\kappa^{6}\Big(\|u_1(t^{n+1})\|^{6}_I+\|u_{h,2}^n\|^{6}_I+h^{-2}\Big)\|\phi_{h,1}^n\|^2_{\Omega_1}+\kappa^{6}\bigg(\|u_1(t^{n+1})\|^{6}_I+\|u_{h,2}^n\|^{6}_I\nonumber\allowdisplaybreaks\\&&+||[{\bf u}(t^{n+1})]||^{6}_I\bigg)\|\phi_{h,2}^n\|^2_{\Omega_2}+\kappa^{6}\|u_{h,2}^n\|^{6}_I\Big(\|{\bf \phi}_{h,1}^{n-1}\|^2_{\Omega_1}+\|{\bf \phi}_{h,2}^{n-1}\|^2_{\Omega_2}\Big)\allowdisplaybreaks
	\Bigg).\label{err10}
	\end{eqnarray}
	%Rewriting \eqref{err10} yields
	Under the interpolation estimates \eqref{inp2} and \eqref{inp}, the terms on the right hand side of \eqref{err10} can be expressed as
	\begin{eqnarray}
	\int_{t^n}^{t^{n+1}}\|\partial_t\eta_{1}^{n+1}\|_{\Omega_1}^2&\leq&h^{2k+2}\|\partial_{t}u_1 \|_{L^2(t^n,t^{n+1};H^{k+1}(\Omega))}^2\label{err22},\\[5pt]\|\nabla \eta_1^{n+1}\|_{\Omega_1}^2&\leq&h^{2k}\|u_1\|_{k+1}^2,\\[5pt]
	\|\eta_{2}^n\|_{\Omega_1}^2&\leq&h^{2k+2}\|u_2\|_{k+1}^2,\\[5pt]
	||[\boldsymbol{\eta}^n]||^2_I&\leq&h^{2k+2}||[{\bf u}]||^2_{k+1},\\[5pt]
	\|p_1(t^{n+1})-q_1^{n+1}\|_{\Omega_1}^2&\leq&h^{2k}\|p_1\|_{k}^2, \label{pq}\\[5pt]
\Delta t\|\partial_{t}u_1-\frac{{u_1(t^{n+1})}-u_{1}(t^{n})}{\Delta t} \|_{\Omega_1}^2&\leq&\Delta t^2\|\partial_{tt}u_1\|_{L^\infty(t^n,t^{n+1};L^2(\Omega))}^2,\label{err23}
	%||[\boldsymbol{\eta}^{n-1}]||^2_I&\leq&h^{2k+1}||[{\bf u}^{n-1}]||^2_{I}\\
	\end{eqnarray}	
%	{\color{red} Some kind of validation is needed for the last inequality above. It either needs to be proven, or a reference must be provided to a VMS paper, where this result could be found.}
	%\begin{eqnarray}
	%\|u_1(t^{n+1})\|_I^2||[{\bf u}(t^n)]|-|[{\bf u}(t^{n+1})]||^2_I+\|u_{2}(t^{n+1})-u_{2}(t^{n})\|_I^2||[{\bf u}(t^{n+1})]||^2_I\nonumber\\\allowdisplaybreaks\frac{1}{2}\|u_{h,2}^n\|_I\Big(|[{\bf u}(t^{n})-{\bf u}(t^{n+1})]|+|[{\bf u}(t^{n-1})-{\bf u}(t^{n+1})]|\Big)&\leq&\nonumber
	%\end{eqnarray}
	%\begin{eqnarray}
	%\|\phi_{h,1}^{n+1}\|^2\|\nabla u_{1}(t^{n+1})\|^4&\leq&\nonumber\\
	%\|\eta_{1}^{n+1}\|_I^2||[{\bf u}^{n}]||^2_I\kappa^4\|u_1(t^{n+1})\|^4_I\Big(\|\phi_{h,1}^n\|^2_{\Omega_1}+\|\phi_{h,2}^n\|^2_{\Omega_2}\nonumber\|\phi_{h,1}^{n+1}\|^2_{\Omega_1}\Big)&\leq&\nonumber\\\Big(\|\nabla \phi_{h,1}^n\|^2_{\Omega_1}+\|\nabla \phi_{h,2}^n\|^2_{\Omega_2}\Big)&\leq&\nonumber\\
	%\kappa^4||[{\bf u}(t^{n+1})]||_I^4(\|\phi_{h,2}^{n}\|^2_{\Omega_2}+\|\phi_{h,1}^{n+1}\|^2_{\Omega_1})&\leq&\nonumber\\
	%\kappa^4\|u_{h,2}^n\|^4_I\Big(\|{\bf \phi}_{h,1}^n\|^2_{\Omega_1}+\|{\bf \phi}_{h,2}^n\|^2_{\Omega_2}+\|{\bf \phi}_{h,1}^{n-1}\|^2_{\Omega_1}\nonumber\nonumber+\|{\bf \phi}_{h,2}^{n-1}\|^2_{\Omega_2}+\|\phi_{h,1}^{n+1}\|^2_{\Omega_1}\Big)&\leq&\nonumber\\\Big(\|\nabla{\bf \phi}_{h,1}^n\|^2_{\Omega_1}+\|\nabla{\bf \phi}_{h,2}^n\|^2_{\Omega_2}\nonumber+\|\nabla {\bf \phi}_{h,1}^{n-1}\|^2_{\Omega_1}+\|\nabla {\bf \phi}_{h,2}^{n-1}\|^2_{\Omega_2}+\|\nabla \phi_{h,1}^{n+1}\|^2_{\Omega_1}\Big)&\leq&
	%\end{eqnarray}
	Substituting  \eqref{err22}-\eqref{err23} into \eqref{err10} and summing over the time steps yield
		\begin{eqnarray}
	\lefteqn{{\|\phi_{h,1}^{M+1}\|_{\Omega_1}^2-\|\phi_{h,1}^{1}\|_{\Omega_1}^2} +\sum_{n=1}^{M}\|\phi_{h,1}^{n+1}-\phi_{h,1}^{n}\|_{\Omega_1}^2 +(\nu_1+\nu_T)\Delta t\sum_{n=1}^{M}\|\nabla  \phi_{h,1}^{n+1}\|_{\Omega_1}}
	\nonumber\\&&+2\kappa\Delta t \sum_{n=1}^{M}\int_{I}|[{\bf u}_h^n]|  |\phi_{h,1}^{n+1}|^2ds-\frac{{ \nu_1+\nu_T}}{8}\Delta t\sum_{n=1}^{M}(\|\nabla{\bf \phi}_{h,1}^n\|^2_{\Omega_1}+\|\nabla {\bf \phi}_{h,1}^{n-1}\|^2_{\Omega_1})\nonumber\allowdisplaybreaks\\&&-\frac{{ \nu_2+\nu_T}}{8}\Delta t\sum_{n=1}^{M}(\|\nabla{\bf \phi}_{h,2}^n\|^2_{\Omega_2}+\|\nabla {\bf \phi}_{h,2}^{n-1}\|^2_{\Omega_2})\nonumber\allowdisplaybreaks\\&\leq& C\Bigg(h^{2k+2}\|\partial_{t}u_1 \|_{L^2(0,T;H^{k+1}(\Omega_1))}^2+h^{2k}(1+|||\nabla u_{1}|||_{\infty,\Omega_1}^2+S_M)|||u_1|||_{2,k+1}^2\nonumber\\&&+h^{2k}|||p_1|||_{2,k}^2+\Delta t^2(\|\partial_{tt}u_1\|_{L^\infty(0,T;L^2(\Omega_1))}^2+\|\partial_t u_1\|_{L^{\infty}(0,T;H^1(\Omega))}^2)+H^{2k}|||u_1|||_{2,k+1}^2\nonumber\\[5pt]&&+h^{2k+2}(|||u_1|||_{\infty,I}^2+S_M)|||[\boldsymbol{u}]|||^2_{2,k+1}
	+h^{2k+2}(S_M|||u_{1}|||_{2,k+1}^2+||[{\bf u}]||^2_{\infty,I}|||u_{2}|||_{2,k+1}^2)\nonumber\\[5pt]
	&&+\Big(|||\nabla u_{1}|||_{\infty,\Omega_1}^{4}+\kappa^{6}(|||u_1|||^{6}_{\infty,I}+||[{\bf u}]||^{6}_{\infty,I}+S_M)\Big)\Delta t\sum_{n=1}^{M}\|\phi_{h,1}^{n+1}\|_{\Omega_1}^2\nonumber\\&&+\kappa^{6}\Big(|||u_1|||^{6}_{\infty,I}+S_M+h^{-2}\Big)\Delta t\sum_{n=1}^{M}\|\phi_{h,1}^n\|^2_{\Omega_1}+\kappa^{6}\bigg(|||u_1|||^{6}_{\infty,I}+S_M\nonumber\allowdisplaybreaks\\&&+||[{\bf u}]||^{6}_{\infty,I}\bigg)\Delta t\sum_{n=1}^{M}\|\phi_{h,2}^n\|^2_{\Omega_2}+\kappa^{6}S_M\Delta t\sum_{n=1}^{M}\Big(\|{\bf \phi}_{h,1}^{n-1}\|^2_{\Omega_1}+\|{\bf \phi}_{h,2}^{n-1}\|^2_{\Omega_2}\Big)\allowdisplaybreaks
	\Bigg)\label{err}
	\end{eqnarray}
	where $S_M$ represents to right hand side of \eqref{stb}.	Simplifying \eqref{err}, we have
	\begin{eqnarray}
	\lefteqn{{\|\phi_{h,1}^{M+1}\|_{\Omega_1}^2-\|\phi_{h,1}^{1}\|_{\Omega_1}^2} +\sum_{n=1}^{M}\|\phi_{h,1}^{n+1}-\phi_{h,1}^{n}\|_{\Omega_1}^2 +(\nu_1+\nu_T)\Delta t\sum_{n=1}^{M}\|\nabla  \phi_{h,1}^{n+1}\|_{\Omega_1}}
	\nonumber\\&&+2\kappa\Delta t \sum_{n=1}^{M}\int_{I}|[{\bf u}^n]|  |\phi_{h,1}^{n+1}|^2ds-\frac{{ \nu_1+\nu_T}}{8}\Delta t\sum_{n=1}^{M}(\|\nabla{\bf \phi}_{h,1}^n\|^2_{\Omega_1}+\|\nabla {\bf \phi}_{h,1}^{n-1}\|^2_{\Omega_1})\nonumber\allowdisplaybreaks\\&&-\frac{{ \nu_2+\nu_T}}{8}\Delta t\sum_{n=1}^{M}(\|\nabla{\bf \phi}_{h,2}^n\|^2_{\Omega_2}+\|\nabla {\bf \phi}_{h,2}^{n-1}\|^2_{\Omega_2})\nonumber\allowdisplaybreaks\\&\leq& C(\Delta t^2+h^{2k}+H^{2k})+\Big(|||\nabla u_{1}|||_{\infty,\Omega_1}^{4}+\kappa^{6}(|||u_1|||^{6}_{\infty,I}+||[{\bf u}]||^{6}_{\infty,I}\nonumber\\&&+S_M)\Big)\Delta t\sum_{n=1}^{M}\|\phi_{h,1}^{n+1}\|_{\Omega_1}^2+\kappa^{6}\Big(|||u_1|||^{6}_{\infty,I}+S_M+h^{-2}\Big)\Delta t\sum_{n=1}^{M}\|\phi_{h,1}^n\|^2_{\Omega_1}\nonumber\\&&+\kappa^{6}\Big(|||u_1|||^{6}_{\infty,I}+S_M+||[{\bf u}]||^{6}_{\infty,I}\Big)\Delta t\sum_{n=1}^{M}\|\phi_{h,2}^n\|^2_{\Omega_2}\nonumber\\&&+\kappa^{6}S\Delta t\sum_{n=1}^{M}\Big(\|{\bf \phi}_{h,1}^{n-1}\|^2_{\Omega_1}+\|{\bf \phi}_{h,2}^{n-1}\|^2_{\Omega_2}\Big)\allowdisplaybreaks. \label{err24}
	\end{eqnarray}
	Similar to the derivation of \eqref{err24}, we can bound the right hand side of \eqref{err4}. Combining it with \eqref{err24} and using \eqref{stb} give	\begin{eqnarray}
	\lefteqn{{\|\boldsymbol{\phi}_{h}^{M+1}\|^2-\|\boldsymbol{\phi}_{h}^{1}\|^2} +\sum_{n=1}^{M}\|\boldsymbol{\phi}_{h}^{n+1}-\boldsymbol{\phi}_{h}^{n}\|^2 +2\kappa\Delta t \sum_{n=1}^{M}\int_{I}|[{\bf u}^n]|  |\boldsymbol{\phi}_{h}^{n+1}|^2ds}
	\nonumber\\&&+{ \frac{3}{4}(\nu_1+\nu_T)\Delta t\sum_{n=1}^{M}\|\nabla  \phi_{h_1}^{n+1}\|^2 +\frac{3}{4}(\nu_2+\nu_T)\Delta t\sum_{n=1}^{M}\|\nabla  \phi_{h,2}^{n+1}\|^2}
	\nonumber\\&&+\frac{1}{8}({{ \nu_1+\nu_T}})\Delta t\Big(2\sum_{n=1}^{M}(\|\nabla{\bf \phi}_{h,1}^{n+1}\|^2_{\Omega_1}-\|\nabla{\bf \phi}_{h,1}^n\|^2_{\Omega_1})+\sum_{n=1}^{M}(\|\nabla{\bf \phi}_{h,1}^{n}\|^2_{\Omega_1}-\|\nabla {\bf \phi}_{h,1}^{n-1}\|^2_{\Omega_1})\Big)\nonumber\\&&+\frac{1}{8}({{ \nu_2+\nu_T}})\Delta t\Big(2\sum_{n=1}^{M}(\|\nabla{\bf \phi}_{h,2}^{n+1}\|^2_{\Omega_2}-\|\nabla{\bf \phi}_{h,2}^n\|^2_{\Omega_2})+\sum_{n=1}^{M}(\|\nabla {\bf \phi}_{h,2}^{n}\|^2_{\Omega_2}-\|\nabla {\bf \phi}_{h,2}^{n-1}\|^2_{\Omega_2})\Big)\nonumber\allowdisplaybreaks\\\allowdisplaybreaks&\leq& C(\Delta t^2+h^{2k}+H^{2k})\nonumber\allowdisplaybreaks\\&&+\Big(|||\nabla {\bf u}|||_{\infty,\Omega}^{4}+\kappa^{ 6}(||[{\bf u}]||^{ 6}_{\infty,I}+|||\boldsymbol{u}|||^{ 6}_{\infty,I}+S_M)\Big)\Delta t\sum_{n=1}^{M}\|\boldsymbol{\phi}_{h}^{n+1}\|^2\nonumber\allowdisplaybreaks\\&&+\kappa^{ 6}\Big(|||\boldsymbol{u}|||^{ 6}_{\infty,I}+||[{\bf u}]||_{\infty,I}^{ 6}+S_M+h^{-2}\Big)\Delta t\sum_{n=1}^M\|\boldsymbol{ \phi}_{h}^n\|^2+\kappa^{6}S_M\Delta t\sum_{n=1}^M\|\boldsymbol{ \phi}_{h}^{n-1}\|^2.\label{err25}
	\end{eqnarray}
 Dropping the positive term and using discrete Gronwall Lemma \ref{gron} produce
	\begin{eqnarray}
	\lefteqn{\|\boldsymbol{\phi}_{h}^{M+1}\|^2  +2\kappa\Delta t \sum_{n=1}^{M}\int_{I}|[{\bf u}^n]|  |\boldsymbol{\phi}_{h}^{n+1}|^2ds}
	\nonumber\\&&+{ \frac{3}{4}(\nu_1+\nu_T)\Delta t\sum_{n=1}^{M}\|\nabla  \phi_{h_1}^{n+1}\|^2 +\frac{3}{4}(\nu_2+\nu_T)\Delta t\sum_{n=1}^{M}\|\nabla  \phi_{h,2}^{n+1}\|^2}
	\nonumber\\&&+\frac{1}{8}({{ \nu_1+\nu_T}})\Delta t(2\|\nabla{\bf \phi}_{h,1}^{M+1}\|^2_{\Omega_1}+\|\nabla{\bf \phi}_{h,1}^{M}\|^2_{\Omega_1})\nonumber\\&&+\frac{1}{8}({{ \nu_2+\nu_T}})\Delta t(2\|\nabla{\bf \phi}_{h,2}^{M+1}\|^2_{\Omega_2}+\|\nabla {\bf \phi}_{h,2}^{M}\|^2_{\Omega_2})\nonumber\\&\leq& \|\boldsymbol{\phi}_{h}^{1}\|^2+\frac{1}{8}({{ \nu_1+\nu_T}})\Delta t(2\|\nabla{\bf \phi}_{h,1}^1\|^2_{\Omega_1}+\|\nabla {\bf \phi}_{h,1}^{0}\|^2_{\Omega_1})\nonumber\\&&+\frac{1}{8}({{ \nu_2+\nu_T}})\Delta t(2\|\nabla{\bf \phi}_{h,2}^1\|^2_{\Omega_2}+\|\nabla {\bf \phi}_{h,2}^{0}\|^2_{\Omega_2})+C(\Delta t^2+h^{2k}+H^{2k}).
	\end{eqnarray}
	Applying triangle inequalities yields the stated result of the theorem.
	%\begin{eqnarray}
	%\lefteqn{\|{\bf u}(t^{n+1})-{\bf u}^{n+1}\|^2  +(\frac{3\nu}{4}+\nu_T)\Delta t\sum_{n=1}^{M}\|\nabla  ({\bf u}(t^{n+1})-{\bf u}^{n+1})\|^2}\nonumber\\&&+\kappa\Delta t \sum_{n=1}^{M}\int_{I}|[{\bf u}^n]|  |{\bf u}(t^{n+1})-{\bf u}^{n+1}|^2ds
	%\nonumber\\&\leq& \|{\bf u}(t^{1})-{\bf u}^{1}\|^2+\|\nabla({ u_1}(t^{1})-u_{h,1}^{1})\|^2_{\Omega_1}+\|\nabla({ u_{h,2}}(t^{1})-u_{h,1}^{2})\|^2_{\Omega_2}+\|\nabla ({ u_1}(t^{0})-u_{h,1}^{0})\|^2_{\Omega_1}\nonumber\\&&+\|\nabla ({ u_{h,2}}(t^{0})-u_{h,2}^{0})\|^2_{\Omega_2}+C(\Delta t^2+h^{2k}+H^{2k})
	%\end{eqnarray}
	
\end{proof}

\begin{cor} \label{corr}
Let $(\bu,{\bf p})$be a solution of \eqref{eq:atmo}-\eqref{eq:atmoBC} with regularity assumptions \eqref{r} and suppose that $(X_i^h,Q_i^h)$ for $i=1,2$ is given by $P_2/P_1$ Taylor Hood finite elements and $L_i^H$ is given by $P_1$ polynomials, $\nu_T=h$ and $H=h$. Assume the velocity data
$\bu^0$, $\bu^1$ satisfies $$\|\bu(t^0)-\bu^0\|_X+\|\bu(t^1)-\bu^1\|_X \leq C_1h$$
for a generic constant $C_1$ independent of $\Delta t$ and $h$. Then, the error satisfies,
\begin{gather}
\|{\bf u}(t^{{M}+1})-{\bf u}^{{M}+1}\|^2  +{ \frac{3}{4}(\nu_1+\nu_T)\Delta t\sum_{n=1}^{M}\|\nabla  (u_1(t^{n+1})-u_{h,1}^{n+1})\|^2}\nonumber
\\
+\frac{3}{4}(\nu_2+\nu_T)\Delta t\sum_{n=1}^{M}\|\nabla  (u_2(t^{n+1})-u_{h,2}^{n+1})\|^2
\leq C((\Delta t)^2+h^{2k}).\label{cor}
	\end{gather}
\end{cor}

\section{Numerical Studies}

In this section, we present a couple of numerical experiments which illustrate situations in which GA-VMS method discussed in the previous sections is beneficial. Numerical studies of GA-VMS method includes a comparison with different types of finite element discretizations of AO interaction. The first experiment serves as a support for the orders of convergence given by Corollary \ref{corr}. The second experiment is to show the energy balance of the problem. The last includes the flow over a cliff type of problem. The simulations were performed with the Taylor-Hood pair of spaces $(P2/P1)$ for velocity and pressure, and also piecewise linear finite element space $P1$ for the large scale space on the same mesh instead of piecewise quadratic finite element space $P2$ on a different coarse mesh, see \cite{jokaya1}; otherwise requires transfer of solutions from one mesh to the other which adds extra computational complexity.

We first compare our results with GA of \cite{CHL12}. The scheme reads: Find $(u_{h,i}^{n+1},p_{h,i}^{n+1})\in (X_{i}^h,Q_{i}^h)$ satisfying
\begin{eqnarray}
(\frac{{u^{n+1}_{h,i}}-u_{h,i}^{n}}{\Delta t} ,v_{h,i})_{\Omega_i} +\nu_i(\nabla  u_{h,i}^{n+1},\nabla v_{h,i})_{\Omega_i}
+({u}_{h,i}^{n+1}\cdot \nabla {u}_{h,i}^{n+1},v_{h,i})_{\Omega_i} \nonumber
-(p_{h,i}^{n+1},\nabla \cdot v_{h,i}) \nonumber\\
+(\nabla\cdot {u}_{h,i}^{n+1},q_{h,i})_{\Omega_i}
+\kappa \int_{I}|[{\bf u}_h^n]|u_{h,i}^{n+1}v_{h,i}ds
-\kappa \int_{I}u_{h,j}^n|[{\bf u}_h^n]|^{1/2}[{\bf u}_h^{n-1}]^{1/2}v_{h,i}ds,
=(\bff_i^{n+1},v_{h,i})_{\Omega_i}\label{algg3}
\end{eqnarray}
for all $(v_{h,i},q_{h,i}) \in (X_{i}^h,Q_{i}^h)$.

In addition, we also use monolithically coupled algorithms for comparison. This way, the proposed model could be compared against computationally very expensive, yet highly accurate, in terms of interface coupling, solutions. TWM and TWM-VMS refer to solving the system two-way monolithically and two-way monolithically with variational multiscale method, respectively. Galerkin FEM approximation of TWM method reads: Find $(u_{h,i}^{n+1},p_{h,i}^{n+1})\in (X_{i}^h,Q_{i}^h)$ satisfying

\begin{eqnarray}
(\frac{{u^{n+1}_{h,i}}-u_{h,i}^{n}}{\Delta t} ,v_{h,i})_{\Omega_i} +
\nu_i(\nabla  {u}_{h,i}^{n+1},\nabla v_{h,i})_{\Omega_i}
+({u}_{h,i}^{n+1}\cdot \nabla {u}_{h,i}^{n+1},v_{h,i})_{\Omega_i}
-(p_{h,i}^{n+1},\nabla \cdot v_{h,i}) \nonumber
\\
+(\nabla\cdot {u}_{h,i}^{n+1},q_{h,i})_{\Omega_i}
+\kappa \int_{I}|[ {\bf u}_h^n]|[{\bf u}_h^{n+1}]v_{h,i}ds
=(\bff_i^{n+1},v_{h,i})_{\Omega_i},\label{BE2}
\end{eqnarray}
for all $(v_{h,i},q_{h,i}) \in (\bX_{i}^h,Q_{i}^h)$.
Similar to GA-VMS method, TWM-VMS finite element discretization reads: Find $(u_{h,i}^{n+1},p_{h,i}^{n+1},\mathbb{G}_i^{\mathbb{H},{n+1}})\in (X_{i}^h,Q_{i}^h,L_i^H)$ satisfying
\begin{eqnarray}
(\frac{{u^{n+1}_{h,i}}-u_{h,i}^{n}}{\Delta t} ,v_{h,i})_{\Omega_i} +(\nu_i+\nu_T)(\nabla  {u}_{h,i}^{n+1},\nabla v_{h,i})_{\Omega_i} + ({u}_{h,i}^{n+1}\cdot \nabla {u}_{h,i}^{n+1},v_{h,i})_{\Omega_i}-(p_{h,i}^{n+1},\nabla \cdot v_{h,i}) \nonumber
\\
 +(\nabla\cdot {u}_{h,i}^{n+1},q_{h,i})_{\Omega_i}+\kappa \int_{I}|[{\bf u}_h^n]|[{\bf u}_h^{n+1}]v_{h,i}ds
=(\bff_i^{n+1},v_{h,i})_{\Omega_i}+\nu_T(\mathbb{G}_i^{\mathbb{H},n},\nabla v_{h,i}), \label{BE3}\\
(\mathbb{G}_i^{\mathbb{H},n}-\nabla u_{h,i}^n, \mathbb{L}_i^H)_{\Omega_i}=0, \label{algg2}
\end{eqnarray}
for all $(v_{h,i},q_{h,i},\mathbb{L}_i^H) \in (\bX_{i}^h,Q_{i}^h,L_i^H)$.

Simulations were  performed at a problem defined in $\Omega = \Omega_1 \cup \Omega_2$ with $\Omega_1 = [0,1] \times [0,1]$ and $\Omega_2=[0,1] \times [0,-1]$ with prescribed solution
\begin{equation*}
\begin{aligned}
u_{1,1} &= a\nu_1e^{-2bt}x^2(1 - x)^2(1 + y) + ae^{-bt}x(1 - x)\nu_1/\sqrt{\kappa a}\\
u_{1,2} &= a\nu_1e^{-2bt}xy(2 + y)(1 - x)(2x-1) + ae^{-bt}y(2x - 1)\nu_1/\sqrt{\kappa a} \\
u_{2,1} &= a\nu_1 e^{-2bt}x^2(1 - x)^2(1 + \frac{\nu_1}{\nu_2}y) \\
u_{2,2} &= a\nu_1 e^{-2bt}xy(1 - x)(2x - 1)(2 + \frac{\nu_1}{\nu_2}y),
\end{aligned}
\end{equation*}
Herein, for simplicity pressures are set to zero in both domains, and right hand side forcing, boundary and two initial values are computed using the manufactured true solution as is done in \cite{ACEL18}. Problem parameters, $b=1/2$, $\kappa=0.001$ and the final time $T=1$ are fixed while $a$, $\nu_1$ and $\nu_2$ vary from one computation to the other. Numerical experiments are performed on a single mesh, that is $H=h$.  Also discretization parameters, $h$, $\Delta t$ and the eddy viscosity parameter $\nu_T=h$ are refined all together. Therefore, first order accuracy is expected in numerical experiments.

{\it Convergence Rates.} Results with the high-viscosity are presented in Tables \ref{table:GA}-\ref{table:TWM-VMS}. These results agree with the analytical predictions in terms of accuracy. That means, decoupling systems and neglecting unresolved small scales will not impose significantly high error. On contrary, latter might improve accuracy even for high viscosities, see $L^2$-norm-in-space and $L^2$-norm-in-time errors in Table \ref{table:TWM} and Table \ref{table:TWM-VMS}. Note that when it comes to low-viscosity results, GA and TWM both fail to converge since small viscosity causes numerical singularities. On the other hand, equipping GA and TWM with VMS regularizes their systems and produces believable results for higher viscosity, see Table \ref{table:GA-VMSsmall}-\ref{table:TWM-VMSsmall} for the choices $\nu_1 = 0.0005$, $\nu_2 = 0.0001$, $a = 1/\nu_1$. Altogether, the behavior of the discrete solutions observed here is in agreement
with the analytical results: GA-VMS is a first order accuracy model of atmosphere-ocean interaction. It can be also observed that decoupling systems will not introduce too much error as TWM-VMS and GA-VMS both give very similar accuracy results. This might be attributed to the dominating viscosity error (instead of decoupling error).

\begin{table}[!ht]
	\caption{GA for $\nu_1 = 0.5$, $\nu_2 = 0.1$, $a = 1$}
	\begin{center} \label{table:GA}
		\begin{tabular}{|c|c|c|c|c|c|c|c|}
			\hline N  & $||\,u-u^h\,||_{L^2(0,T;L^2(\Omega))}$ & rate & $||\,u-u^h\,||_{L^2(0,T;H^1(\Omega))}$ & rate\\
			\hline 8  & 1.14578e-03 & -    & 1.24305e-02 & -    \\
			\hline 16 & 5.73429e-04 & 1.00 & 4.86981e-03 & 1.35 \\
			\hline 32 & 2.87691e-04 & 1.00 & 2.25678e-03 & 1.11 \\
			\hline 64 & 1.44198e-04 & 1.00 & 1.10762e-03 & 1.03 \\
			\hline
		\end{tabular}
	\end{center}
\end{table}

\begin{table}[!ht]
	\caption{GA-VMS for $\nu_1 = 0.5$, $\nu_2 = 0.1$, $a = 1$}
	\begin{center} \label{table:GA-VMS}
		\begin{tabular}{|c|c|c|c|c|c|c|c|}
			\hline N  & $||\,u-u^h\,||_{L^2(0,T;L^2(\Omega))}$ & rate & $||\,u-u^h\,||_{L^2(0,T;H^1(\Omega))}$ & rate\\
			\hline 8  & 1.76862e-03 & -    & 1.64437e-02 & -    \\
			\hline 16 & 7.39919e-04 & 1.26 & 6.11638e-03 & 1.43 \\
			\hline 32 & 3.29011e-04 & 1.17 & 2.58223e-03 & 1.24 \\
			\hline 64 & 1.54366e-04 & 1.09 & 1.18872e-03 & 1.12 \\
			\hline
		\end{tabular}
	\end{center}
\end{table}
\begin{table}[!ht]
	\caption{TWM for $\nu_1 = 0.5$, $\nu_2 = 0.1$, $a = 1$}
	\begin{center} \label{table:TWM}
		\begin{tabular}{|c|c|c|c|c|c|c|c|}
			\hline N  & $||\,u-u^h\,||_{L^2(0,T;L^2(\Omega))}$ & rate & $||\,u-u^h\,||_{L^2(0,T;H^1(\Omega))}$ & rate\\
			\hline 8  & 1.09092e-03 & -    & 1.19716e-02 & -    \\
			\hline 16 & 5.46568e-04 & 1.00 & 4.58332e-03 & 1.39 \\
			\hline 32 & 2.74340e-04 & 1.00 & 2.10125e-03 & 1.13 \\
			\hline 64 & 1.37532e-04 & 1.00 & 1.02956e-03 & 1.03 \\
			\hline
		\end{tabular}
	\end{center}
\end{table}
\begin{table}[!ht]
	\caption{TWM-VMS for $\nu_1 = 0.5$, $\nu_2 = 0.1$, $a = 1$}
	\begin{center} \label{table:TWM-VMS}
		\begin{tabular}{|c|c|c|c|c|c|c|c|}
			\hline N  & $||\,u-u^h\,||_{L^2(0,T;L^2(\Omega))}$ & rate & $||\,u-u^h\,||_{L^2(0,T;H^1(\Omega))}$ & rate\\
			\hline 8  & 1.56615e-03 & -    & 1.56615e-02 & -    \\
			\hline 16 & 6.96009e-04 & 1.27 & 5.68593e-03 & 1.46 \\
			\hline 32 & 2.38204e-04 & 1.17 & 2.38261e-03 & 1.25 \\
			\hline 64 & 1.09747e-04 & 1.09 & 1.09747e-03 & 1.12 \\
			\hline
		\end{tabular}
	\end{center}
\end{table}

\begin{table}[!ht]
	\caption{GA-VMS for $\nu_1 = 0.0005$, $\nu_2 = 0.0001$, $a = 1/\nu_1$}
	\begin{center} \label{table:GA-VMSsmall}
		\begin{tabular}{|c|c|c|c|c|c|c|c|}
			\hline N  & $||\,u-u^h\,||_{L^2(0,T;L^2(\Omega))}$ & rate & $||\,u-u^h\,||_{L^2(0,T;H^1(\Omega))}$ & rate\\
			\hline 8  & 1.01687e-02 & -    & 8.89222e-02 & -    \\
			\hline 16 & 4.26050e-03 & 1.26 & 4.53765e-02 & 0.97 \\
			\hline 32 & 1.49500e-03 & 1.51 & 2.29722e-02 & 0.98 \\
			\hline 64 & 5.21601e-04 & 1.52 & 1.18369e-02 & 0.96 \\
			\hline 128& 1.98533e-04 & 1.39 & 5.58328e-03 & 1.08 \\
			\hline
		\end{tabular}
	\end{center}
\end{table}

\begin{table}[!ht]
	\caption{TWM-VMS for $\nu_1 = 0.0005$, $\nu_2 = 0.0001$, $a = 1/\nu_1$}
	\begin{center} \label{table:TWM-VMSsmall}
		\begin{tabular}{|c|c|c|c|c|c|c|c|}
			\hline N  & $||\,u-u^h\,||_{L^2(0,T;L^2(\Omega))}$ & rate & $||\,u-u^h\,||_{L^2(0,T;H^1(\Omega))}$ & rate\\
			\hline 8  & 1.01681e-02 & -    & 8.89157e-02 & -    \\
			\hline 16 & 4.25999e-03 & 1.26 & 4.53667e-02 & 0.97 \\
			\hline 32 & 1.49464e-03 & 1.51 & 2.29561e-02 & 0.98 \\
			\hline 64 & 5.21364e-04 & 1.52 & 1.18120e-02 & 0.96 \\
			\hline 128& 1.98402e-04 & 1.39 & 5.55465e-03 & 1.09 \\
			\hline
		\end{tabular}
	\end{center}
\end{table}
It has to be noted that the alternative approach for GA-VMS mentioned on the Remark \ref{remark:alternative_approach} fails to provide good-quality results, see Table \ref{table:GA-VMS_alt_small}.
\begin{table}[!ht]
	\caption{GA-VMS alternative approach for $\nu_1 = 0.0005$, $\nu_2 = 0.0001$, $a = 1/\nu_1$}
	\begin{center} \label{table:GA-VMS_alt_small}
		\begin{tabular}{|c|c|c|c|c|c|c|c|}
			\hline N  & $||\,u-u^h\,||_{L^2(0,T;L^2(\Omega))}$ & rate & $||\,u-u^h\,||_{L^2(0,T;H^1(\Omega))}$ & rate\\
			\hline 8  & 1.53291e-02 & -    & 1.64229e-01 & -    \\
			\hline 16 & 9.03667e-03 & 0.76 & 1.56901e-01 & 0.07 \\
			\hline 32 & 5.55247e-03 & 0.70 & 1.81161e-01 & -0.21 \\
			\hline 64 & 3.65918e-03 & 0.60 & 2.18847e-01 & -0.27 \\
			\hline
		\end{tabular}
	\end{center}
\end{table}

{\it Conservation of Energy.} Computational results related to conservation of global energy is presented next. For simplicity, the problem has been set to keep the same total energy over all the time levels. For this reason, we choose homogeneous Dirichlet boundary conditions everywhere except the interface and zero forcing. Expectations of the energy has to be drawn before giving any results. To that end, weak formulation of the continuous problem shall be considered under homogeneous boundary conditions and zero forcing: Multiplying \eqref{eq:atmo} by $\textbf{u} \in X$, integrating over the whole domain and over $[0,T]$, we get the following energy equality:
\begin{equation}\label{cont_energy}
||\textbf{u}(t)||_{\Omega}^2 + 2\nu\int_{0}^t||\nabla \textbf{u}(s)||_{\Omega}^2ds =  ||\textbf{u}(0)||_{\Omega}^2.
\end{equation}
Herein, define
\begin{equation}\nonumber
\begin{split}
I:=\text{initial kinetic energy } &= ||\textbf{u}(0)||_{\Omega}^2 = ||u_1(0)||_{\Omega_1}^2 + ||u_2(0)||_{\Omega_2}^2,\\
KE:=\text{kinetic energy at t } &= ||\textbf{u}(t)||_{\Omega}^2 = ||u_1(t)||_{\Omega_1}^2 + ||u_2(t)||_{\Omega_2}^2,\\
\mathcal{E}:=\text{energy dissipated by the time t } &= 2\nu\int_{0}^t||\nabla \textbf{u}(s)||_{\Omega}^2ds = 2\nu_1 \int_{0}^t||\nabla u_1(s)||_{\Omega_1}^2ds
\\
&+ 2\nu_2 \int_{0}^t||\nabla u_2(s)||_{\Omega_2}^2ds.
\end{split}
\end{equation}
Energy equality \eqref{cont_energy} means continuous system conserves energy for all time. However, discrete models introduce discretization error such as decoupling errors, consequently, energy is not exactly conserved. The following quantity gives a measurement of how far away energy goes beyond being exact. Considering discrete versions of energies, define
\begin{eqnarray}
AED(t):=\text{absolute energy difference} &=& |I-KE-\mathcal{E}|
\end{eqnarray}
As mentioned above for continuous solution, the problem has been constructed so that it has zero forcing and homogeneous Dirichlet boundary conditions everywhere except the interface, and divergence-free initial values, $u_{i,j}^0$ have been chosen as follows:
\begin{equation}\label{IC}
\begin{split}
u_{i,1} &= \sin(2\pi y)\sin^2(\pi x),\\
u_{i,2} &=-\sin(2\pi x)\sin^2(\pi y), i=1,2.
\end{split}
\end{equation}
Both GA and GA-VMS require two initial values. Therefore, we compute the second initial values with one step of IMEX method proposed in \cite{CHL12} and investigated in \cite{ZHS16}.
\begin{figure}[H]
	\includegraphics[width=0.5\linewidth]{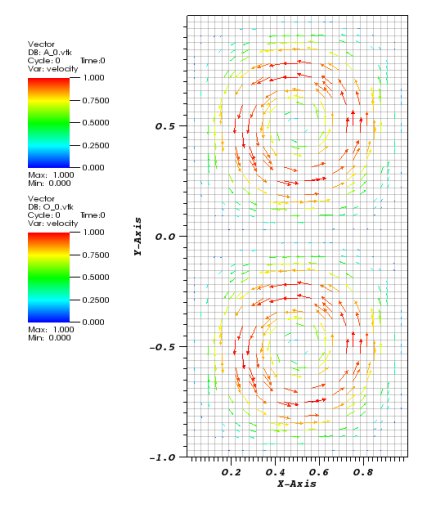}
	\caption{Initial flows}
	\label{fig:IV}
\end{figure}
Discretization parameters are chosen uniform, $h=1/32$, $\Delta t = 0.01$, and computations ended at the final time $T=25$. Problem parameters, $\kappa = 0.001$, $\nu_1 = 1.5e-03$ and $\nu_2 = 1.0e-04$ have been chosen, and computations have been performed on the uniform square mesh shown faded in Figure \ref{fig:IV}. It has to be noted that these choice of parameters is very close to being realistic in terms of drag coefficient $\kappa$ and the ratio of the viscosities. Totally realistic setting with real viscocities causes very prohibitive singularities in GA, at this point, we increase the values of viscosities for reliable GA results. In addition, even under this choices, computations take much longer time for linear systems of GA to converge as seen in the Table \ref{table:times}.
\begin{table}[!ht]
	\caption{Computational times}
	\begin{center} \label{table:times}
		\begin{tabular}{|c|c|c|c|c|c|c|c|}
			\hline GA  &  4h:13m:58s   \\
			\hline GA-VMS & 41m:04s   \\
			\hline
		\end{tabular}
	\end{center}
\end{table}
Noting the fact that global energy is exactly conserved in the true solution of AO interaction, any proposed model shall conserve it as much as possible. Although the mathematical definitions of the energy and energy dissipation rates in GA and GA-VMS are different from continuous formulation, their solutions both physically approach the same quantity, true solution, therefore, any well-constructed comparison should be made with physical meanings of energies given in \eqref{cont_energy}.
\begin{figure}[H]
	\includegraphics[width=0.5\linewidth]{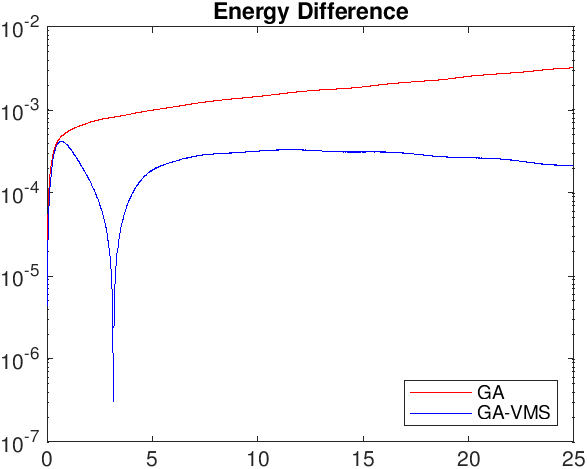}
	\caption{Absolute energy differences between the total energy and the initial energy input, $AED(t)$}
	\label{fig:Energy_difference}
\end{figure}
The absolute differences between the total energy and the initial energy input are computed over all the time levels, and presented in the Figure \ref{fig:Energy_difference}. Clearly, GA-VMS performs better than GA in terms of conservation of the total energy.

In addition, Figure \ref{fig:IV} shows that initial values are inversely rotating flows on both domains and differ only in directions. As a result, only the interaction on the interface determines their expectancy. It can be noted that the flow with higher viscosity will decay faster, due to higher dissipation. Consequently, energy transfer is expected to happen from the domain with the low-viscosity flow to high-viscosity flow, in a long-enough run. Figure \ref{fig:total_energies_separately} illustrates that this expectation has been met by GA-VMS since the total energy in the atmosphere increases beyond the initial energy input while the exact opposite happens in the ocean. On the other hand, the total energy with GA immediately starts dropping in both domains, yet still keeping higher total energy in the atmosphere but less than the initial energy input, which means energy transfer from the ocean to the atmosphere has lost within the numerical error(if ever resolved correctly). This is an obvious achievement for GA-VMS since the goal of such models is to resolve energy transfer reliably.
\begin{figure}[H]
	\includegraphics[width=0.5\linewidth]{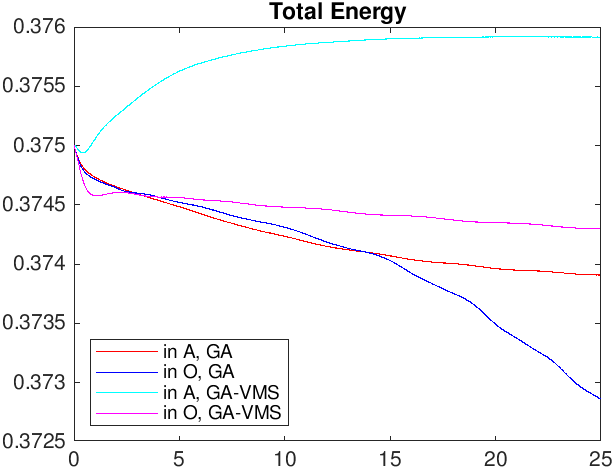}
	\caption{Total energies, ($KE+\mathcal{E}$) of Atmosphere(A) and Ocean(O), separately.}
	\label{fig:total_energies_separately}
\end{figure}

{\it  Long-Time Stability.} We now present computational results for the long-time stability of GA and GA-VMS will be given for a problem, that is constructed so that a parabolic inflow in the atmosphere passes a backward-facing step | a widely used benchmark problem for one-domain fluid-flow | before atmosphere and ocean met, see the domain in the Figure \ref{fig:domain}. Note that this step could be a coast mountain, cliff, etc. in a real life simulation.
\begin{figure}[H]
	\includegraphics[width=0.75\linewidth]{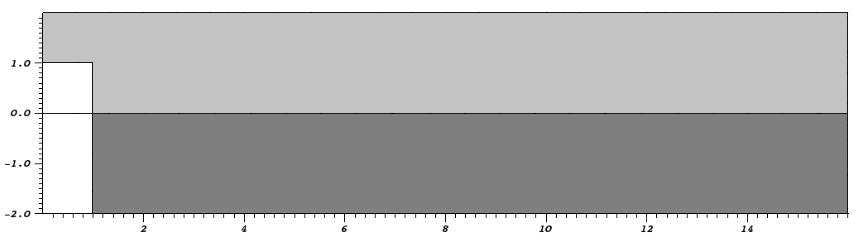}
	\caption{Domain}
	\label{fig:domain}
\end{figure}
Homogeneous Dirichlet boundary conditions have been strongly enforced on the step, on the left wall and the bottom of the ocean. While parabolic inflow profile with maximum inlet 1 drives the flow in the atmosphere,``do nothing'' boundary conditions weakly imposed on the outflow, on the top of atmosphere and the right wall of the ocean. Both fluids are in rest initially, and the second initial values have been computed by one-step of IMEX method as in the previous example, i.e. flows in both domains start with the same initial values. Rest of the parameters have been chosen as in Table \ref{table:tableparameters}.
\begin{table}[!ht]
	\caption{Problem parameters}
	\begin{center} \label{table:tableparameters}
		\begin{tabular}{|c|c|c|c|c|c|c|c|}
			\hline $\nu_1$  &  $\nu_2$  & $\kappa$ & T   & $\Delta t$  & h  & $\nu_T$       \\
			\hline   5e-04  &    5e-03  & 2.45e-03 & 100 &   0.01      & 0.1 - 0.14  & 0.01      \\
			\hline
		\end{tabular}
	\end{center}
\end{table}

Figure \ref{fig:stability} illustrates that the solution with GA starts blowing up around $t=25$ while GA-VMS produces stable results all the way up to final time $T=100$.

\begin{figure}[!ht]
	\includegraphics[width=0.5\linewidth]{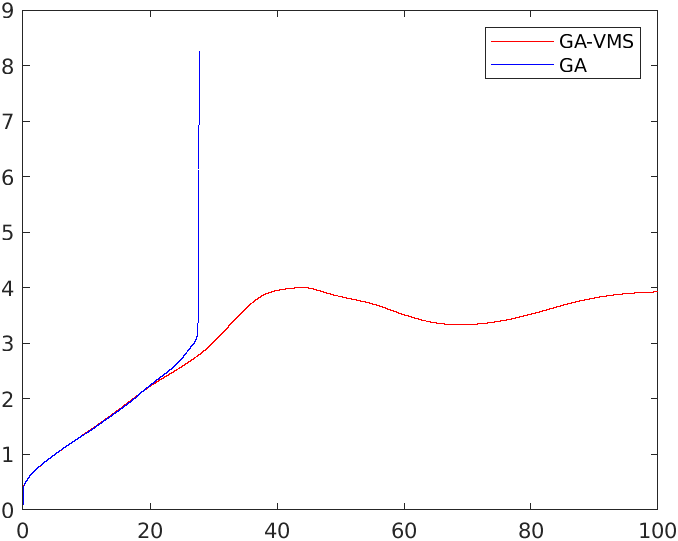}
	\caption{Temporal evolution of $||u_{h}^{n+1}||$ with GA and GA-VMS}
	\label{fig:stability}
\end{figure}

Expected vector fields with GA and GA-VMS (in Figure \ref{fig:step_comparison}) illustrate that both methods produce very similar results as long as they are both stable. However, as seen in the Figure \ref{ga25} and Figure \ref{ga27.75}, solution with GA has already started blowing up around $t=25$.

Figures \ref{fig:step_comparison} and \ref{fig:step_vms} suggest that the interface flow in the ocean tends to follow the direction of the flow just above. For this reason, all consistent direction changes on the interface of the atmosphere results in a separate vortex formation right below. Furthermore, the reattachment point in the atmospheric flow and the separation point of two vertices in the ocean coincide. One can intuitively expect this phenomenon already since, for this setting, the oceanic flow is due to merely its interaction with the atmosphere and possess of very low energy to determine its own persistent direction.

\begin{figure}[htp]
\begin{center}
\subfigure[GA at t=10]{\label{ga10}\includegraphics[width=0.4\textwidth]{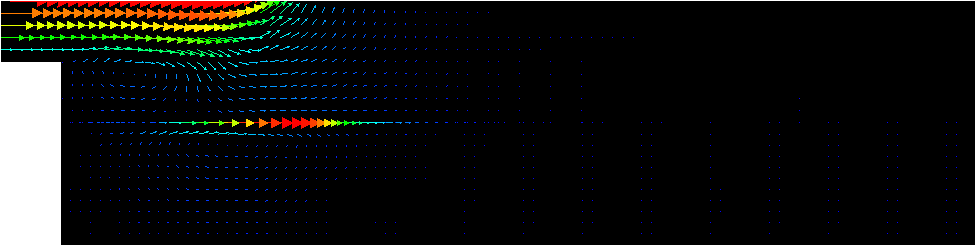}}
\subfigure[GA-VMS at t=10]{\label{vms20}\includegraphics[width=0.4\textwidth]{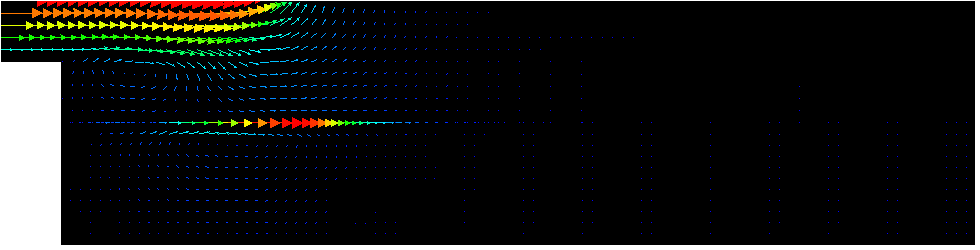}}  \\
\subfigure[GA at t=20]{\label{ga20}\includegraphics[width=0.4\textwidth]{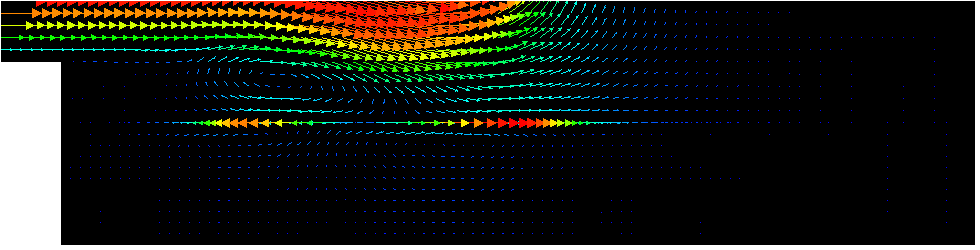}}
\subfigure[GA-VMS at t=20]{\label{vms20}\includegraphics[width=0.4\textwidth]{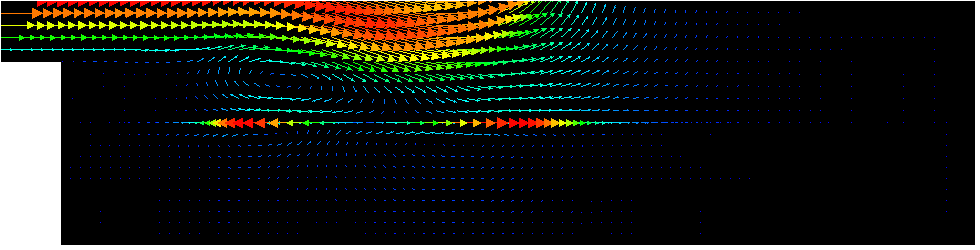}} \\
\subfigure[GA at t=25]{\label{ga25}\includegraphics[width=0.4\textwidth]{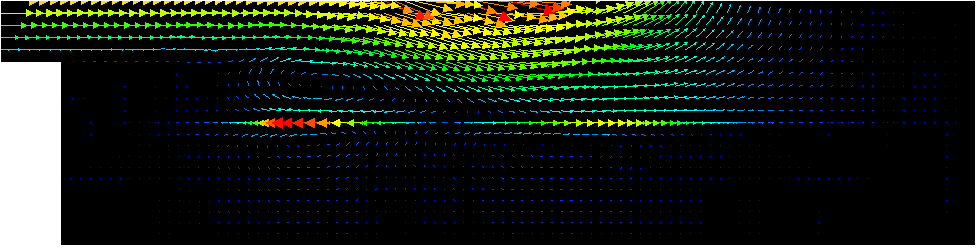}}
\subfigure[GA-VMS at t=25]{\label{vms25}\includegraphics[width=0.4\textwidth]{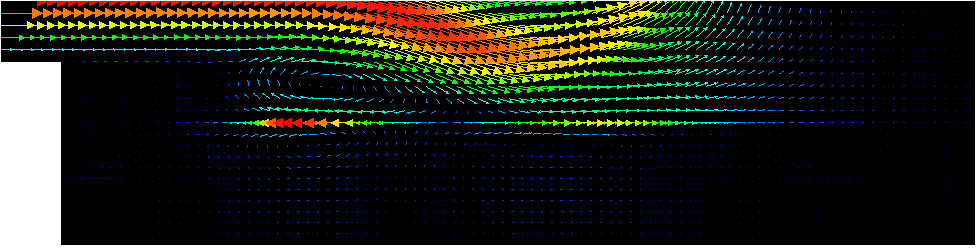}} \\
\subfigure[GA at t=27.75]{\label{ga27.75}\includegraphics[width=0.4\textwidth]{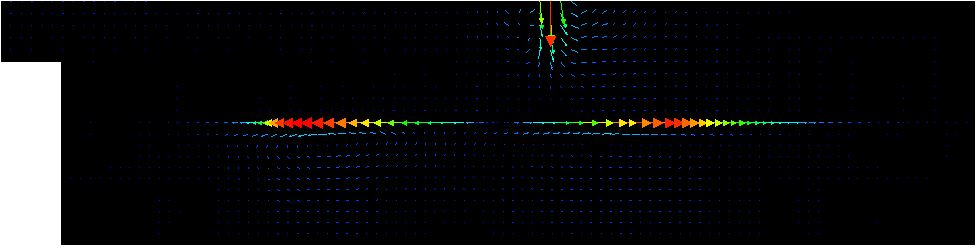}}
\subfigure[GA-VMS at t=27.75]{\label{vms27.75}\includegraphics[width=0.4\textwidth]{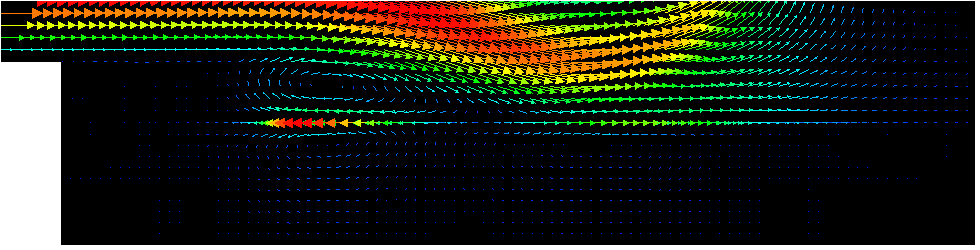}} \\
\end{center}
\caption{Expected vector fields with GA and GA-VMS}
\label{fig:step_comparison}
\vspace*{-0.65cm}
\end{figure}

\begin{figure}[htp]
\begin{center}
\subfigure[GA-VMS at t=30]{\label{vms30}\includegraphics[width=0.4\textwidth]{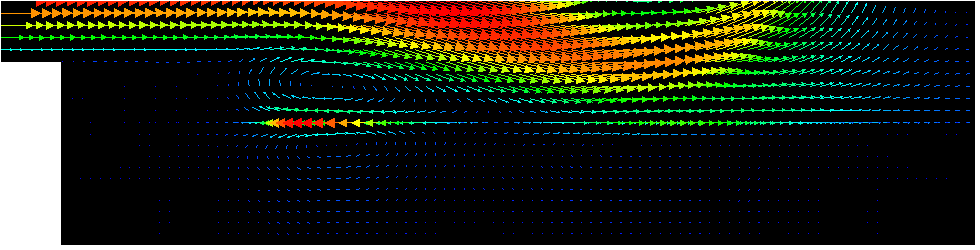}}
\subfigure[GA-VMS at t=40]{\label{vms40}\includegraphics[width=0.4\textwidth]{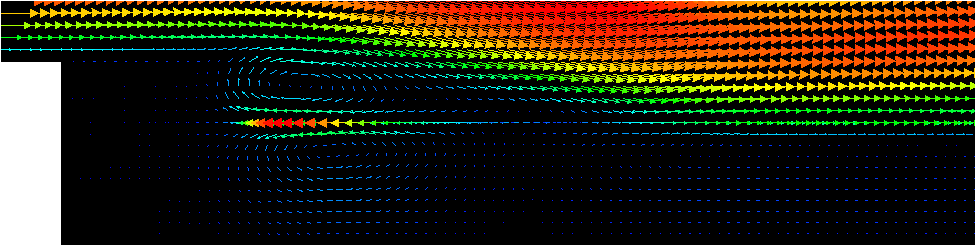}} \\
\subfigure[GA-VMS at t=50]{\label{vms50}\includegraphics[width=0.4\textwidth]{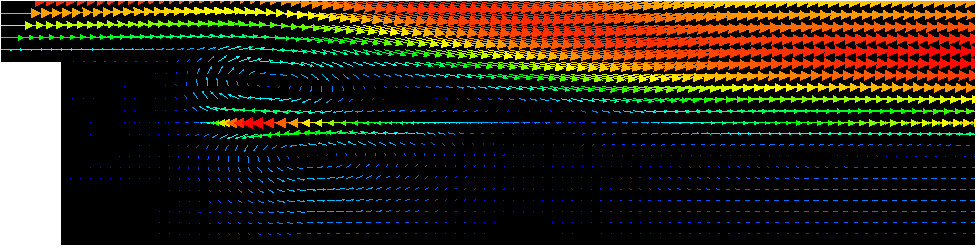}}
\subfigure[GA-VMS at t=60]{\label{vms60}\includegraphics[width=0.4\textwidth]{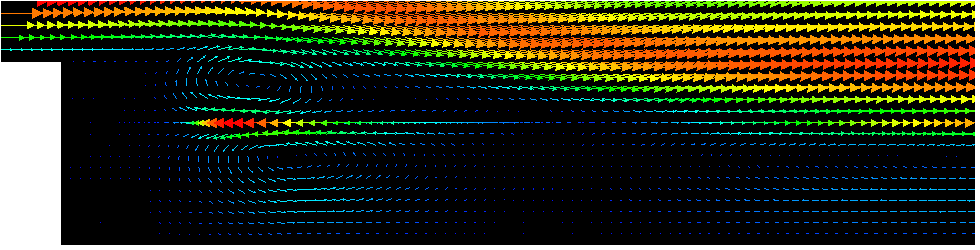}} \\
\subfigure[GA-VMS at t=70]{\label{vms70}\includegraphics[width=0.4\textwidth]{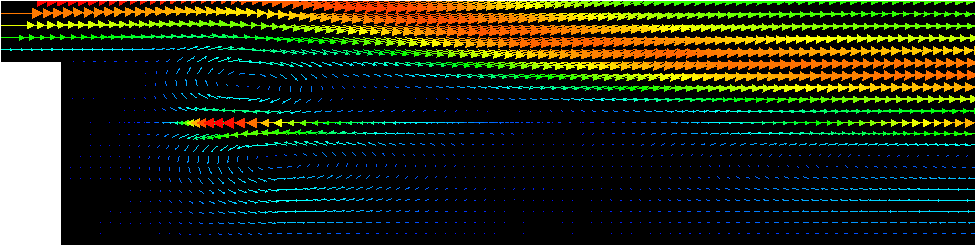}}
\subfigure[GA-VMS at t=80]{\label{vms80}\includegraphics[width=0.4\textwidth]{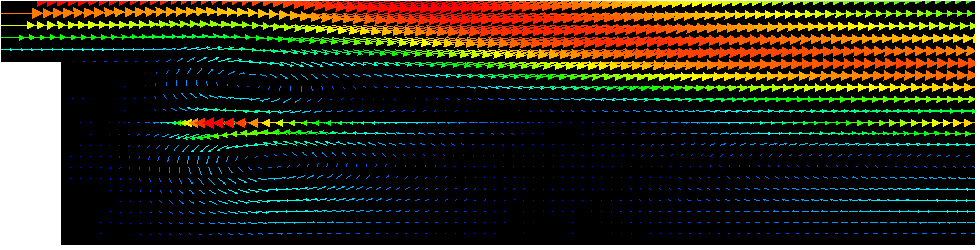}} \\
\subfigure[GA-VMS at t=90]{\label{vms90}\includegraphics[width=0.4\textwidth]{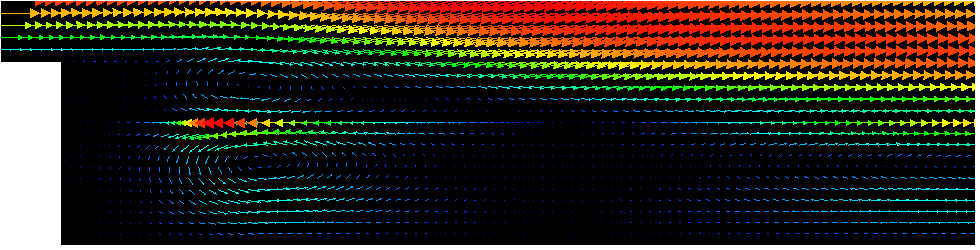}}
\subfigure[GA-VMS at t=100]{\label{vms100}\includegraphics[width=0.4\textwidth]{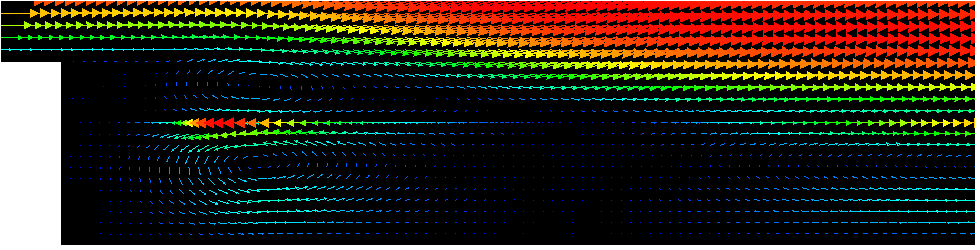}}
\end{center}
\caption{Expected vector fields with GA-VMS}
\label{fig:step_vms}
\vspace*{-0.65cm}
\end{figure}

\section{Conclusions}

In this report we introduced a method for approximating solutions to a turbulent fluid-fluid interaction problem \eqref{eq:atmo}-\eqref{eq:atmoBC}. The method combines the Geometric Averaging method for stable decoupling of the two-domain problem with the Variational Multiscale stabilization technique for high Reynolds number flows. We performed full numerical analysis of the method, proving its stability and accuracy. One of the challenges we had to overcome was the lack of benchmark problems for qualitative testing of our method in the case of low viscosities, $\nu << 1$. In addition to verifying numerically the claimed theoretical accuracy of the method in the case of a known true solution, we also used two other numerical tests to assess the qualitative behavior of the solution. First, we showed that the total global energy of the approximate solution is better conserved with the proposed method - as it should be in the continuous coupled solution. And also, energy transfer from the domain with high energy to the domain with low energy is reliably captured. Secondly, we introduced a ``flow over a cliff'' type of a problem, which could serve as an analogue of flow over a step, in the case of fluid-fluid interaction. The vortices forming and detaching in the air domain were closely matched by the sea regions with increased flow velocity. The GA method (without the VMS component) had failed to work in any of the tests, if the viscosity coefficient was taken to be small enough, while the proposed GA-VMS technique has matched the expectations both quantitatively and qualitatively.

\bibliographystyle{amsplain}
\bibliography{ref}
\end{document}